\documentclass[12pt]{amsart}
\usepackage[top=1in, bottom=1in, left=1in, right=1in]{geometry}
\usepackage[T2A,T1]{fontenc}
\usepackage{times}

\usepackage{amssymb,amsmath,amsthm}
\usepackage{mathrsfs} %to use \mathscr fonts
\usepackage{bbm} %use with \mathbbm command to produce ``double-lined'' letters and numbers
\usepackage[hyphens]{url}
\usepackage{hyperref}
\usepackage{dsfont}
\usepackage{nicefrac}
\usepackage{mathtools}
\usepackage{xspace}
\usepackage{nicematrix}
\usepackage{scalerel}
\usepackage{stackengine,wasysym}

\usepackage{graphicx,tikz}
\usepackage{color}

\usepackage{todonotes}

\makeatletter
\renewcommand\part{\@startsection{part}{2}%
	\z@{0.5\linespacing\@plus2\linespacing}{\linespacing}%
	{\normalfont\large\scshape\bfseries\centering}}
\makeatother

\definecolor{myblue}{rgb}{0.09,0.32,0.44} %22-84-113
\hypersetup{pdfborder={0 0 0},pdfborderstyle={},colorlinks=true,linkcolor=myblue,citecolor=myblue,urlcolor=blue}

\numberwithin{equation}{section}

\theoremstyle{plain}

\theoremstyle{plain}

\newtheorem{theorem}{Theorem}

\newtheorem{lem}{Lemma}[section]

\newtheorem{corollary}[lem]{Corollary}

\newtheorem{lemma}[lem]{Lemma}
\newtheorem{prop}[lem]{Proposition}
\newtheorem{proposition}[lem]{Proposition}

\newtheorem{thm-n}[lem]{Theorem}

\theoremstyle{remark}

\newtheorem*{rem*}{Remark}
\newtheorem*{notat*}{Notation}
\newtheorem*{exm*}{Example}

\theoremstyle{definition}
\newtheorem{dfn}[lem]{Definition}

\newcommand\reallywidetilde[1]{\ThisStyle{%
		\setbox0=\hbox{$\SavedStyle#1$}%
		\stackengine{-.1\LMpt}{$\SavedStyle#1$}{%
			\stretchto{\scaleto{\SavedStyle\mkern.2mu\AC}{.5150\wd0}}{.6\ht0}%
		}{O}{c}{F}{T}{S}%
}}
\newcommand{\bflam}{\lambda}

\newcommand{\ZZ}{\mathbb{Z}}

\renewcommand{\AA}{\mathbb{A}}

\newcommand{\FF}{\F}

\newcommand{\F}{\mathbb{F}}

\newcommand{\Prob}{\mathrm{Prob}}
\newcommand{\Disc}{\mathrm{Disc}}
\newcommand{\bfA}{{\mathbf{A}}}
\newcommand{\bfa}{{\mathbf{a}}}

\newcommand{\CD}{\mathcal{D}}

\newcommand{\CI}{\mathcal{I}}
\newcommand{\CJ}{\mathcal{J}}

\newcommand{\CS}{\mathcal{S}}

\DeclareMathOperator{\ord}{ord}

\DeclareMathOperator{\Spec}{Spec}
\DeclareMathOperator{\chr}{char}
\DeclareMathOperator{\GL}{GL}
\DeclareMathOperator{\SL}{SL}

\newcommand{\Gal}{{\mathrm{Gal}}}
\newcommand{\monic}{{\mathrm{monic}}}

\renewcommand{\tilde}{\widetilde}

\let\ogphi\phi
\renewcommand{\phi}{\varphi}

\newcommand{\bs}\boldsymbol{}

\newcommand{\Fr}{{\rm Fr}}

\renewcommand{\bar}[1]{\overline{#1}}

\newcommand{\tmatrix}[3]{{\left[\begin{array}{cc}#1&#2\\&#3\end{array}\right]}}
\newcommand{\stmatrix}[3]{{\left[\begin{smallmatrix}{#1}& {#2}\\ {}& {#3} \end{smallmatrix}\right]}}

\makeatletter
\def\moverlay{\mathpalette\mov@rlay}
\def\mov@rlay#1#2{\leavevmode\vtop{%
		\baselineskip\z@skip \lineskiplimit-\maxdimen
		\ialign{\hfil$\m@th#1##$\hfil\cr#2\crcr}}}
\newcommand{\charfusion}[3][\mathord]{
	#1{\ifx#1\mathop\vphantom{#2}\fi
		\mathpalette\mov@rlay{#2\cr#3}
	}
	\ifx#1\mathop\expandafter\displaylimits\fi}
\makeatother

\newcommand{\con}{{\rm con}}

\renewcommand{\subset}{\subseteq}
\renewcommand{\supset}{\supseteq}

\title{Galois groups of random additive polynomials}
\author{Lior Bary-Soroker, Alexei Entin and Eilidh McKemmie}
\date{\today}
\begin{document}

	\maketitle
	
	\begin{abstract}
		We study the distribution of the Galois group  of a random $q$-additive polynomial over a rational function field:
		For $q$ a power of a prime $p$, let $f=X^{q^n}+a_{n-1}X^{q^{n-1}}+\ldots+a_1X^q+a_0X$ be a random polynomial chosen uniformly from the set of $q$-additive polynomials of degree $n$ and height $d$, that is, the coefficients are independent uniform polynomials of degree $\deg a_i\leq d$. The Galois group $G_f$ is a random subgroup of $\GL_n(q)$. 
		Our main result shows that $G_f$ is almost surely large as $d,q$ are fixed and $n\to \infty$. For example, we give necessary and sufficient conditions so that $\SL_n(q)\leq G_f$ asymptotically almost surely. Our proof uses the classification of maximal subgroups of $\GL_n(q)$.
		We also consider the
		limits: $q,n$ fixed, $d\to \infty$ and $d,n$ fixed, $q\to \infty$, which are more elementary.
	\end{abstract}
	
	\section{Introduction}
	
	The Galois theory of random polynomials is a classical area of study going back to Hilbert and van der Waerden. Recently, there is a renewed interest in this area. In the so-called large box model, in which the coefficients are chosen uniformly from a large box of integers whose size tends to infinity, recent  results were obtained by Anderson et al \cite{andersonquantitative}, Chow and Dietmann \cite{chow2021towards} and Bhargava \cite{bhargava2021galois}. In the restricted coefficients (or small box) model, in which the coefficients are bounded (or chosen from a finite set), and the degree is tending to infinity, recent results were obtained by the first author, Kozma and Koukoulopoulos \cite{bary2020irreducible, bary2020irreducibility} and Breuillard and Varj\'u \cite{breuillard2019irreducibility}. The function field analogue of this problem was studied by the second author jointly with Popov \cite{EP}. See also \cite{eberhard2022characteristic, ferber2022random} for recent results on the characteristic polynomial of a random matrix in different ensembles.
	
	The present work studies random additive polynomials over the function field $\F_q(t)$. Let $p$ be a prime and $q$ a power of $p$. Let  $K$ be a field containing $\mathbb{F}_q$; in particular, $\chr (K) = p$. A polynomial $f\in K[X]$ is called \emph{additive} (resp.\ \emph{$q$-additive}) if $f$ induces a linear map (resp.\ $\mathbb{F}_q$-linear map) on the algebraic closure $\bar{K}$ of $K$.  
	Additive polynomials play a central role in arithmetic of global function fields, and one motivation comes from understanding the distribution of their Galois groups. Another motivation is derived from the fact that over the integers the Galois groups are typically the full permutation group, and, in contrast, the roots of additive polynomials have extra structure, and hence the Galois groups are never the full symmetric group.

	The polynomial $f$ is $q$-additive if and only if it has the form  $f(X) =  a_n X^{q^n} + a_{n-1} X^{q^{n-1}} + \cdots+ a_0X$; $f$ is separable if and only if $a_0\neq 0$; and the set of its roots is a vector space over $\FF_q$. 
	Assume $f$ is separable and let $G_f$ be the Galois group of $f$ over $K$. Then, the action of $G_f$ on the roots respects the $\FF_q$-linear structure, and so we have a faithful representation  $G_f\leqslant \GL_n(q)$. Galois groups are only defined up to conjugation, but it is convenient to neglect this in the notation. So, for example, we will write $G_f=H$, if they are conjugate. 
	
	We are interested in random $q$-additive polynomials $f(t,X)\in \FF_q[t][X]$ and the distribution of $G_f$.  More precisely, denote by $\FF_q[t]_{\leq d}$ the set of polynomials with coefficients in $\FF_q$ of degree $\le d$. Let $a_0, \ldots, a_{n-1}$ be independent random variables taking 
	values in $\FF_q[t]_{\leq  d}$ uniformly, and let 
	\[
	f(X) =   X^{q^n} + a_{n-1} X^{q^{n-1}} + \cdots + a_0X
	\]
	be a random $q$-additive monic polynomial. If $a_0\neq 0$, then we have a random subgroup $G_f\leqslant \GL_n(q)$. There are three natural limits: The large box model in which $n,q$ are fixed and $d\to \infty$, the large finite field model in which $d,n$ are fixed and $q\to \infty$, and the restricted coefficient (or small box) model in which $q,d$ are fixed and $n\to \infty$. The large box model follows easily from the generic case (in which the coefficients are variables \cite{Dickson_1911,wilkerson1983primer}) combined with Hilbert's Irreducibility Theorem:  $G_f = \GL_n(q)$ asymptotically almost surely, see Section~\ref{sec:largebox} for details. In Section~\ref{sec:large q} we show the same in the large finite field model, which is somewhat more challenging.
	
	The restricted coefficient model is significantly more challenging than the other models\footnote
	{The large box and large finite field models become much more challenging if one attempts to identify all Galois groups that occur with probability $\ge q^{-d}$, similarly to \cite{bhargava2021galois}. We do not pursue this finer investigation in the present work.}: it does not follow from the generic case, there are several groups that occur with positive probability, and it necessitates the development of a novel approach that incorporates deep results from group theory. 
	Given a polynomial $f(t,X)\in \FF_q[t,X]$, we define  
	$\con_t(f)\in \FF_q[X]$ to be the greatest common divisor of the coefficients of $f$ when considered as a polynomial in $t$, that is, the largest monic divisor of $f$ in $\FF_q[X]$. For an additive polynomial necessarily $X\,|\,\con_t(f)$. If $h=\con_t(f)\neq X$, then the roots of $h$ form a nontrivial invariant subspace of $G_f$, in particular, $G_f\not\geqslant \SL_n(q)$. Our first result says that $\con_t(f)\neq X$ is (almost surely) the only obstructions for a large Galois group.
	
	\begin{theorem}\label{thm:1}
		Fix $d>0$ and $q$ a prime power. 
		Let $a_0,\ldots a_{n-1}$ be independent random variables, taking values in $\FF_q[t]_{\le d}$ uniformly. Let  $f = X^{q^n} + a_{n-1} X^{q^{n-1}} + \cdots + a_0 X$ and let $G_f$ be the Galois group of $f$ over $\FF_q(t)$. Then 
		\[
		\lim_{n\to \infty} {\rm Prob} (G_f \geqslant \SL_n(q)  \mid \con_t(f)=X) = 1.  
		\]
		%If we further condition on $a_0 = cu^k$ with $c\in \mathbb{F}_q^\times$ and $k$ maximal, then asymptotically almost surely $G_f =\{ g\in \GL_n(q) : \det g \in \left<c\right> \FF_q^{\times k}\}$ as $n\to \infty$.
	\end{theorem}
	
	We note that with positive probability, we have that $h=\con_t(f)\neq X$. 
	Then, $h\in \FF_q[X]$ is also a $q$-additive polynomial (see Corollary \ref{cor:content additive}), we write  $h = X^{q^{\eta}} + h_{\eta -1}X^{q^{\eta-1}} + \cdots + h_0 X$ and we let 
	\begin{equation}\label{Hmatrix}
		D = {\begin{pmatrix}0&&&&-h_{0}\\1&& &&-h_{1}\\ & &\ddots & & \vdots \\&& &1&-h_{{\eta-1}}\end{pmatrix}}
	\end{equation}
	be the companion matrix. Then, the roots of $h$ form a $G_f$-invariant subspace of dimension $\eta = \log_q \deg h$ and  the action of $G_f$ restricted to this subspace is via the Frobenius automorphism $\Fr_q$, which is given explicitly by the matrix $D$ (in a suitable basis). Therefore, $G_f$ is contained in the group of matrices of the form $\left(\begin{smallmatrix}\left<D\right> & *\\&\GL_{n-\eta}\end{smallmatrix}\right)$. It turns out that the typical Galois group depends on $a_0$ and $h=\con_t(f)$:
	
	\begin{theorem}\label{thm:2}
		Let $f$ be a random $q$-additive polynomial, as in Theorem~\ref{thm:1}.
		Let $h=\sum_{i=0}^\eta h_iX^{q^i}\in \FF_q[X]$ be a fixed $q$-additive polynomial with $h_0\neq 0$, and $D$ its companion matrix as in \eqref{Hmatrix}. Let $k\geq 0$ be an integer, $u\in\F_q[t],1\le\deg u\le d/k$ monic and not a power of another polynomial, $c\in \FF_q^\times$. 
		Let $\Gamma_{n,h,c,k}\leqslant \GL_n(q)$ be the subgroup 
		\begin{equation}\label{eq:def gamma}
			\Gamma_{n,h,c,k}=\left\{\left(\begin{smallmatrix}D^i & A\\0&B\end{smallmatrix}\right): i\in\ZZ,A\in M_{\eta\times(n-\eta)}(\F_q),B\in\GL_{n-\eta}(q),\det(B)\in\frac{c^i}{\det D^i} \FF_q^{\times k}\right\}.
		\end{equation}
		% such that $A=H^i$ for some $i$ and  $\det(C)\in \frac{c^i}{\det H^i}\FF_q^{\times k}$ (or $\det \gamma \in c^i \FF_q^{\times k}$. 
		Then 
		\[
		\lim_{n\to \infty} \Prob(G_f=\Gamma_{n,h,c,k} \mid \con_t(f)=h, a_0=cu^k) =1
		\]
		(here the equality $G_f=\Gamma_{n,h,c,k}$ is with respect to a basis as described in the paragraph preceding the theorem).
		
		% Fix $d>0$ and $q$ a prime power. 
		% Let $a_0,\ldots a_{n-1}$ be independent random variables, taking the values $\FF_q[t]_{<d}$ uniformly. Let  $f = X^{q^n} + a_{n-1} X^{q^{n-1}} + \cdots + a_0 X$ and let $G_f$ be the Galois group of $f$ over $\FF_q(t)$. Let $h=X^{q^r}+ h_{r-1}X^{q^{r-1}}+ \cdots + h_0X$. We will condition on $h=\con_t(f)$ and $a_0=cu^k$ where $u\in \FF_q[t]$ is a non-constant polynomial and $c\in \FF_q$ and $k$ is maximal.
		
		% Define \[H=\begin{pmatrix}-h_{r-1} & -h_{r-2}& \cdots & -h_0\\
			% 1&&&\\
			% &\ddots&&\\
			% &&1&0\end{pmatrix}.\]
		
		% Then asymptotically almost surely \[G_f=\left\{\left(\begin{smallmatrix}A & B\\&C\end{smallmatrix}\right) \in \GL_n(q) \,\middle|\, A=H^i \mbox{ for some }i, \det(C)\frac{h_0^i}{c^i}\in \FF_q^{\times k}\right\}.\]
	\end{theorem}
	
	Given $h,c,k,u$ as in Theorem~\ref{thm:2} it is easy to compute the asymptotic probability of the event $\con_t(f)=h,a_0=cu^k$ conditional on $a_0\neq 0$, and it is easy to show that these asymptotic probabilities over all $h,c,k,u$ sum to 1, using $u=t$ and $k=0$ in the case that $a_0 \in \FF_q^\times$. Hence a random $f$ with $a_0\neq 0$ falls outside the scope of Theorem~\ref{thm:2} with asymptotic probability 0. 
	Theorem \ref{thm:2} has the following
	
	\begin{corollary}\label{cor:irreduc} Let $f$ be a random $q$-additive polynomial as in Theorem \ref{thm:1}. Then
		$$\lim_{n\to\infty}\Prob(f/ \con_tf\mbox{ is irreducible})=1.$$
	\end{corollary}
	
	The corollary is trivial in the case $d=1$ ($f$ is linear in $t$), however the case $d>1$ seems to require almost the full force of Theorem \ref{thm:2}.
	
	\subsection{Outline of proofs and organization of the paper}
	The first key step in proving Theorem~\ref{thm:1}, is to show that under the assumption $\con_t(f)=X$, almost surely $G_f\leqslant GL_n(q)$ is an irreducible subgroup. This is equivalent to $f$ not having nontrivial $q$-additive divisors, which is accomplished by a height argument. The second key step is to consider the specialization $f_0(X)=f(0,X)$. It turns out that $G_f$ contains an element $\sigma$ (obtained as a Frobenius element over the prime ideal $t$) with characteristic polynomial $f_0$. Since $f_0$ is essentially uniform (in the set of degree $n$ monic polynomials in $\F_q[X]$), $G_f$ contains elements with "random" characteristic polynomials. However, by deep results in group theory, any irreducible $G\leqslant\GL_n(q)$ not containing $\SL_n(q)$ has characteristic polynomials with very specific structure that rules out most polynomials of degree $n$. Results of Garzoni and Eberhard \cite{eberhard2023conjugacy} and Garzoni and McKemmie \cite{Garzoni} based on bounds of Fulman and Guralnick \cite{FulmanGuralnick4,fulman2012bounds} play a central role in carrying out this argument. This step is the deepest part of our results.
	
	To obtain Theorem~\ref{thm:2}, we slightly refine the methods outlined above and use two main additional ingredients: the (well-known) connection between $\det(G_f)$ and the coefficient $a_0$ (for which we give a self-contained proof since we need a slightly refined statement) is used to fully capture the upper left and lower right blocks of $G_f$ in (\ref{eq:def gamma}). Some linear algebra and results of Higman \cite{Hig62} and Pollatsek \cite{Pol71} on the vanishing of certain group cohomologies allow us to capture the upper-right block of $G_f$ in (\ref{eq:def gamma}).
	
	The paper is organized as follows: in Section~\ref{sec:prelim}, we give the necessary background on additive polynomials. Our exposition is mostly (but not fully) self-contained and covers mostly (but not exclusively) well-known material, but for which a comprehensive source convenient for our applications is unavailable (as far as we know). In section~\ref{sect:irred} we prove that almost surely $G_f$ is almost irreducible, with all proper invariant subspaces contained in $Z(\con_t(f))$, which is typically small. In Section~\ref{sect:subgroups}, we prove a proposition implying that most polynomials $f_0\in\F_q[X],\deg f_0=n$ are not characteristic polynomials of elements contained in irreducible $G\leqslant\GL_n(q)$ not containing $\SL_n(q)$, mostly building on \cite{eberhard2023conjugacy,FulmanGuralnick4,fulman2012bounds}. In Section 5, we combine the tools developed in Sections~\ref{sect:irred} and \ref{sect:subgroups} to conclude the proof of Theorem~\ref{thm:1} and capture the lower right block of $G_f$ in Theorem~\ref{thm:2} up to the determinant. Then in Sections~\ref{sec:det} and \ref{sec:upper right block}, we develop the tools for capturing the diagonal blocks and upper right blocks of $G_f$ respectively. In Section~\ref{sec:assembly}, we assemble all the pieces from the previous sections to conclude the proof of Theorem~\ref{thm:2} and also derive Corollary \ref{cor:irreduc}. In Sections~\ref{sec:largebox} and \ref{sec:large q}, we discuss the large box and large $q$ models respectively. Finally, Appendices~\ref{sect:spec_fact} and \ref{sec:subgroups large q} contain auxiliary results which are likely to have other applications outside the scope of this paper. Appendix~\ref{sect:spec_fact} discusses some Bertini-type results that allow us to find specializations with certain prescribed characteristic polynomials needed in the large $q$ model. Appendix~\ref{sec:subgroups large q} counts characteristic polynomials of elements contained in irreducible $G\leqslant\GL_n(q)$ not containing $\SL_n(q)$ in the case of large $q$, mostly building on \cite{Garzoni,fulman2013number}. 
	
	{\bf Acknowledgments.} We would like to thank Daniele Garzoni for spotting a few errors in a previous draft of this paper and for other valuable comments and to Mihran Papikian for spotting another small error. We would also like to thank the anonymous referee of a previous draft of this paper for their suggestions for improving the exposition. The first author was partially supported by the Israel Science Foundation grant no. 702/19.
	The second author was partially supported by the Israel Science Foundation grant no. 2507/19.
	
	\section{Preliminaries on additive polynomials}
	\label{sec:prelim}
	
	In the present section, we review the background on additive polynomials that we will need in what follows. For a modern treatment of the Galois theory of additive polynomials based on the notion of Frobenius modules see \cite[\S V]{MaMa18}. A more elementary introduction to additive polynomials with a view towards applications to class field theory over function fields appears in \cite[\S 1]{Gos98}. Our treatment will be relatively elementary and self-contained. While most, though not all, of the results in the present section are well-known, we summarize them in a form that is convenient for our applications. 
	
	\subsection{$q$-additive polynomials}\label{sec:q-additive}
	Throughout the present section $p$ is a prime number and $q$ is a power of $p$. Let $F$ be a field of characteristic $p$. A polynomial $f\in F[X]$ is called \emph{$q$-additive} if it has the form 
	$$f=a_nX^{q^n}+a_{n-1}X^{q^{n-1}}+\ldots+a_1X^q+a_0X.$$
	Evidently, if $f$ is $q$-additive then
	\begin{equation}\label{eq:additivity}f(\alpha+\beta)=f(\alpha)+f(\beta),\quad f(a\alpha)=af(\alpha),\quad \alpha,\beta\in \overline F,a\in\F_q.\end{equation}
	Conversely, it is a simple exercise to show that if $f$ satisfies (\ref{eq:additivity}) then it is $q$-additive.
	
	\begin{notat*} For a polynomial $f\in F[X]$ we denote by
		$$Z(f)=\{\alpha\in\overline F:f(\alpha)=0\}$$ its set of roots in a fixed algebraic closure $\overline F$.\end{notat*}
	
	A characterization of $q$-additive polynomials in terms of their roots is given by the following
	
	\begin{prop}\label{prop:basic} Let $f=a_nX^{q^n}+a_{n-1}X^{q^{n-1}}+\ldots+a_0X\in F[X]$ be a $q$-additive polynomial.
		\begin{enumerate}
			\item $Z(f)$ is an $\F_q$-linear subspace of $\overline F$ of dimension $n$.
			\item  $f$ is separable if and only if $a_0\neq 0$.
			\item  If $g\in F[X]$ is separable and $Z(g)$ is an $\F_q$-linear subspace of $\overline F$ then $g$ is $q$-additive.
		\end{enumerate}
	\end{prop}
	
	\begin{proof}\textit{(i)} is immediate from (\ref{eq:additivity}), \textit{(ii)} from the fact that $f'=a_0$. For the proof of \textit{(iii)} see \cite[Corollary 1.2.2]{Gos98}.\end{proof}
	
	\begin{lem}\label{lem:gcd} Let $f,g\in F[X]$ be $q$-additive polynomials. Then $\gcd(f,g)$ is $q$-additive.\end{lem}
	
	\begin{proof} First assume that one of $f,g$ is separable, say $f$. By Proposition~\ref{prop:basic}(i), $Z(f),Z(g)$ are $\F_q$-linear subspaces of $\overline F$ and hence $Z(\gcd(f,g))=Z(f)\cap Z(g)$ is also an $\F_q$-linear subspace. By Proposition~\ref{prop:basic}(iii), $\gcd(f,g)$ is $q$-additive.
		
		If $f,g$ are inseparable, we may write $f(X)=f_1(X^{q^k}),\,g(X)=g_1(X^{q^k})$ with $f_1,g_1$ being $q$-additive and separable and $k\ge 1$. Then by the separable case $\gcd(f,g)=\gcd(f_1,g_1)(X^{q^k})$ is $q$-additive.\end{proof}
	
	If $k$ is a field, the \emph{content} of a bivariate polynomial $f\in k[t,X]=\sum_{i=0}^d c_i(X)t^i,\,c_i\in k[X]$ is defined to be
	$$\con_tf=\gcd(c_0,\ldots,c_d)\in k[X],$$
	where we always take the monic representative of the greatest common divisor.
	
	\begin{corollary}\label{cor:content additive}
		Let $k$ be a field with $\chr k=p$, $F=k(t)$ the univariate rational function field, and $f\in k[t,X]\subset F[X]$ a $q$-additive polynomial in the variable $X$. Then $\con_tf$ is $q$-additive.
	\end{corollary}
	
	\begin{proof} If we write $f=\sum_{i=0}^d c_i(X)t^i$ then clearly $c_i(X)$ are $q$-additive and by Lemma~\ref{lem:gcd} so is $\con_tf=\gcd(c_0,\ldots,c_d)$.
	\end{proof}
	
	Now, we make the additional assumption that $\F_q\subseteq F$ and let $f=X^{q^n}+a_{n-1}X^{q^{n-1}}+\ldots+a_0X,\,a_0\neq 0$ be monic, $q$-additive, and separable. Let $N= F(Z(f))$ be the splitting field of $f$ and denote by 
	$G_f=\Gal(N/F)$ the Galois group of $f$. Then, $G_f$ acts $\F_q$-linearly on the set of roots $Z(f)$, which is an $\F_q$-linear space of dimension $n$ (Proposition~\ref{prop:basic}(i)), since $G_f$ acts by field automorphisms fixing $\F_q\subseteq F$. % and this action is $\F_q$-linear (by our assumption $F\supset\F_q$ and the fact that $G_f$ acts on the splitting field $F(Z(f))$ by $F$-algebra automorphisms). 
	We identify $G_f$ with its image under the inclusion  $G_f\hookrightarrow \GL(Z(f))\cong\GL_n(q)$.
	
	\begin{lem}\label{lem:factor_inv_subspace} Under the assumptions of the above paragraph, there is a bijective correspondence 
		$$g\mapsto Z(g),\quad W\mapsto\prod_{\alpha\in W}(X-\alpha)$$ between monic $q$-additive factors $g\in F[X]$ of $f$ and $G_f$-invariant $\F_q$-linear subspaces  $W\subseteq Z(f)$.\end{lem}
	
	\begin{proof}
		If $g\in F[X]$ is a monic $q$-additive factor of $f$ then it is separable (because $f$ is separable). As $g\in F[X]$,  $G_f$ preserves $Z(g)$, and hence $Z(g)\subset Z(f)$ is a $G_f$-invariant subspace. Conversely, if $W\subset Z(f)$ is a $G_f$-invariant subspace then each $\sigma\in G_f$ permutes $W$, hence the coefficients of $g(X)=\prod_{\alpha\in W}(X-\alpha)$ are fixed by  $\sigma$. By the Galois correspondence, $g\in F[X]$. By Proposition~\ref{prop:basic}(iii), $g$ is $q$-additive.\end{proof}

	\subsection{Additive polynomials over finite fields}\label{sec:add pol finite}
	Let $q$ be a prime power and consider a $q$-additive polynomial
	$$f=X^{q^n}+a_{n-1}X^{q^{n-1}}+\ldots+a_0X,\quad a_i\in\F_{q^r},\quad a_0\neq 0.$$
	In what follows an important role will be played by the \emph{associated polynomial}:
	\begin{equation}\label{eq:def tilde}\tilde f=X^n+a_{n-1}X^{n-1}+\ldots+a_1X+a_0\in\F_{q^r}[X]\end{equation} and its companion matrix
	\begin{equation}\label{eq:def comp}
		D=\left[\begin{array}{cccc}&&&-a_0\\1&&&-a_1\\&\ddots&&\vdots \\&&1&-a_{n-1}\end{array}\right]\in\GL_n(\F_{q^r}).
	\end{equation}
	
	The Frobenius map $\Fr_{q^r}\colon x\mapsto x^{q^r}$ acts $\F_q$-linearly on the $\F_q$-linear space $Z(f)$ and it generates the Galois group $G_f=\Gal(\F_{q^r}(Z(f))/\F_{q^r})$. The following proposition gives an explicit description of a matrix $C$ representing the action  of $\Fr_{q^r}$ acting on $Z(f)$.
	
	\begin{proposition} \label{prop:frob_finite} 
		Let $f$, $\tilde{f}$, and $D$ be as above.  For a matrix $A\in \GL_n(\overline{\F_q})$,  denote by $A^{(q^i)}$ the result of applying $\Fr_{q^i}$ to each entry of $A$. Then:
		\begin{enumerate}
			\item There exists a basis
			of $Z(f)$ as an $\F_q$-vector space such that the matrix $C\in \GL_n(q)$ representing $\Fr_{q^r}|_{Z(f)}$  is conjugate over $\overline{\F_{q}}$ to
			$$B=D\cdot D^{(q)}\cdot D^{(q^2)}\cdots D^{(q^{r-1})}\in\GL_n(\F_{q^r}).$$
			\item Assume $r=1$. Then, we may choose the basis in (i) such that $C=D$ and the characteristic polynomial of $C$ is $\tilde f$.
			\item The determinant of $\Fr_{q^r}$ acting on $Z(f)$ is $(-1)^{rn}N_{\F_{q^r}/\F_q}(a_0)$, where $N_{\F_{q^r}/\F_q}$ is the norm map.
		\end{enumerate}
	\end{proposition}
	
	\begin{proof} \textit{(i)}
		Let $\alpha_1,\ldots,\alpha_n$ be an $\F_q$-basis of $Z(f)$. Then we have
		$$
		(\alpha_1^{q^r}\,\alpha_2^{q^r}\,\cdots\,\alpha_n^{q^r}) =
		(\alpha_1\,\cdots\,\alpha_n)C
		$$ 
		for the matrix $C\in\GL_n(q)$ of $\Fr_{q^r}$ acting on $Z(f)$ (with respect to the basis $\alpha_1,\ldots\alpha_n$). Applying $\Fr_q$ iteratively we obtain $$(\alpha_1^{q^{r+i}}\,\cdots\,\alpha_n^{q^{r+i}}) =
		(\alpha_1^{q^i}\,\cdots\,\alpha_n^{q^i})C,\quad i\in\ZZ_{\ge 0}.$$
		Combining these relations for $i=0,\ldots,n-1$, we obtain the relation
		\begin{equation}\label{eq:moore1}
			\left[\begin{array}{ccc}\alpha_1^{q^r}&\cdots&\alpha_n^{q^r}\\ \alpha_1^{q^{r+1}} &\cdots& \alpha_n^{q^{r+1}} \\ &\vdots& \\ \alpha_1^{q^{r+n-1}}&\cdots&\alpha_n^{q^{r+n-1}}\end{array}\right]=\left[\begin{array}{ccc}\alpha_1&\cdots&\alpha_n\\ \alpha_1^{q} &\cdots& \alpha_n^{q} \\ &\vdots& \\ \alpha_1^{q^{n-1}}&\cdots&\alpha_n^{q^{n-1}}\end{array}\right]C,
		\end{equation}
		equivalently $C=A^{-1}A^{(q^r)}$ where $A\in\GL_n(\overline{\F_{q}})$ is the matrix appearing on the RHS of (\ref{eq:moore1}). Since $\alpha_1,\ldots,\alpha_n$ are linearly independent over $\F_q$, the matrix $A$ is invertible by \cite[Lemma~1.3.3]{Gos98}.
		
		We have $f(\alpha_i)=0$, hence  $\alpha_i^{q^n}=-a_{n-1}\alpha_i^{q^{n-1}}-\ldots-a_0\alpha_i$. Therefore
		$$\left[\begin{array}{ccc}\alpha_1^q&\cdots&\alpha_n^q\\ 
			\alpha_1^{q^2}&\cdots&\alpha_n^{q^2}\\ &\vdots& \\ 
			\alpha_1^{q^n}&\cdots&\alpha_n^{q^n}\end{array}\right]=
		\left[\begin{array}{cccc} 
			&1&&\\&&\ddots&\\&&&1\\-a_0&-a_1&\cdots&-a_{n-1}\end{array}\right]\left[\begin{array}{ccc}\alpha_1&\cdots&\alpha_n\\ 
			\alpha_1^{q}&\cdots&\alpha_n^{q}\\ &\vdots& \\ 
			\alpha_1^{q^{n-1}}&\cdots&\alpha_n^{q^{n-1}}\end{array}\right],$$
		or in compact form $A^{(q)}=D^TA$. Applying $\Fr_q$ iteratively to this relation, we obtain
		$$A^{(q^r)}=\left(D^{(q^{r-1})}\right)^T\cdots\left(D^{(q)}\right)^TD^TA=B^TA,$$ i.e. $B^T=A^{(q^r)}A^{-1}$. But $B\sim B^T = AA^{-1}A^{(q^r)}A^{-1} =ACA^{-1}\sim C$ ($\sim$ denoting conjugacy over $\overline{\F_q}$), as required.

		\textit{(ii)} By (i), since $r=1$,  the matrix $C$ is conjugate to $D$ over $\overline{\F_{q}}$. Since $C,D\in\GL_n(q)$, by the theory of the rational normal form \cite[\S 12.2]{DuFo04} the matrices $C,D$ are conjugate over $\F_q$. In particular $C$ has the same characteristic polynomial as $D$, which is $\tilde f$.

		\textit{(iii)} By (i) and the multiplicativity of the determinant, we have 
		$$\det B=\prod_{i=0}^{r-1}\det D^{(q^i)}=\prod_{i=0}^{r-1}(-1)^na_0^{q^i}=
		(-1)^{rn}N_{\F_{q^r}/\F_q}(a_0),$$
		as needed.
	\end{proof}
	
	We conclude this subsection with a refinement of Lemma~\ref{lem:gcd} in the special case $F=\F_q$.
	
	\begin{lem}\label{lem:fq additive gcd} Let $f_1,\ldots,f_k\in\F_q[t]$ be $q$-additive polynomials. Then
		$$\reallywidetilde{\gcd(f_1,\ldots,f_k)}=\gcd\left(\tilde f_1,\ldots,\tilde f_k\right).$$
	\end{lem}

	\begin{proof} Since $a^q=a$ for all $a\in\F_q$, the map  $f\mapsto\tilde f$ defines an isomorphism of $\F_q$-algebras $$\mathcal A:=\left(\{f\in\F_q[X]:q\mbox{-additive}\},+,\circ\right)\to\left(\F_q[X],+,\cdot\right)$$ ($\circ$ denotes composition).
		
		Let $g\in\mathcal A$ be such that $\tilde g=\gcd\left(\tilde f_1,\ldots,\tilde f_k\right)$. We want to show $g=\gcd(f_1,\ldots,f_k)$. Write $\tilde f_i=\tilde h_i\tilde g$ ($h_i\in\mathcal A$) and so $f_i=h_i\circ g$. Since $X\mid h_i$, we deduce that  $g\mid f_i$. Hence, $g\mid\gcd(f_1,\ldots,f_k)$.
		
		By the Euclidean algorithm,  $\tilde g=\sum_{i=1}^k\tilde u_i\tilde f_i$ for some $u_i\in\mathcal A$. Thus,  $g=\sum_{i=1}^ku_i\circ f_i$. Again, since $X \mid u_i$, we see that $f_i \mid u_i\circ f_i$, and thus $g$ is a linear combination of the $f_i$ with coefficients in $\F_q[X]$. Hence, $\gcd(f_1,\ldots,f_k) \mid g$. This proves that $g=\gcd(f_1,\ldots,f_k)$, as required.
	\end{proof}
	
	\subsection{Additive polynomials over $\F_q(t)$}
	
	In this subsection we consider $q$-additive polynomials over the field $\F_q(t)$. One of the main tools in the Galois theory of such polynomials is specializing the variable $t$. We begin by reviewing the specialization of a separable monic polynomial
	$f=X^n+a_{n-1}X^{n-1}+\ldots+a_0\in\F_q[t][X].$
	For $\tau\in\overline{\F_q}$, we consider the specialization
	$$f_\tau=f(\tau,X)=X^n+a_{n-1}(\tau)X^{n-1}+\ldots+a_0(\tau)\in\overline{\F_q}[X].$$
	
	\begin{dfn} Let $f$ be as above, $\alpha_1,\ldots,\alpha_n\in\overline{\F_q(t)}$ the roots of $f$, $\mathcal O$ the integral closure of $\F_q[t]$ in $L=\F_q(t,\alpha_1,\ldots,\alpha_n)$. Since $f\in\F_q[t][X]$ is monic, $\alpha_i\in\mathcal O$. Let $\tau\in\overline{\F_q}$. A \emph{specialization map} for $f$ and $\tau$ is a homomorphism of $\F_q$-algebras $\psi\colon \mathcal O\to\overline{\F_q}$ such that $\psi(t)=\tau$.\end{dfn}
	
	The basic relation between $G_f=\Gal(f/\F_q(t))$ and $G_{f_\tau}=\Gal(f_\tau/\F_q(\tau))$ is summarized in the following
	
	\begin{prop}\label{prop:spec} Let $f=X^n+a_{n-1}X^{n-1}+\ldots+a_0\in\F_q[t][X]$ be a monic separable polynomial, $Z(f)=\{\alpha_1,\ldots,\alpha_n\}$, $\tau\in\overline{\F_q}$, $r=[\F_q(\tau):\F_q]$.
		\begin{enumerate}
			\item There exists a (non-unique) specialization map $\psi$ for $f,\tau$.
		\end{enumerate}
		If $\psi$ is a specialization map for $f,\tau$ then
		\begin{enumerate}\addtocounter{enumi}{1}
			\item $\psi(Z(f))=Z(f_\tau)$.
			\item There exists a (non-unique) $\sigma\in G_f=\Gal(f/\F_q(t))$ such that $\psi(\sigma(\alpha))=\psi(\alpha)^{q^r}$ for all $\alpha\in Z(f)$.
		\end{enumerate}
		If we assume additionally that $f_\tau$ is separable then
		\begin{enumerate}\addtocounter{enumi}{3}
			\item $\psi|_{Z(f)}\colon Z(f)\to Z(f_\tau)$ is bijective.
			\item The element $\sigma\in G_f$ in (iii) is unique. Furthermore, the conjugacy class of $\sigma$ depends on $f,\tau$ but not on the choice of $\psi$.
		\end{enumerate}
	\end{prop}
	
	\begin{proof}%\todo{Do we really need to include a proof of this proposition? It essentially appears in any text-book on algebraic number theory (usually the text-books prefer the language of primes rather of homomorphism, but it is easy to go from one to another}
		This is a standard consequence of the basic theory of Galois extensions of function fields, as can be found in \cite[\S 9]{Ros02}. We briefly sketch the derivation.
		
		{\it (i)} Let $P$ be the kernel of the unique $\F_q$-algebra 
		homomorphism $\F_q[t]\to\F_q(\tau)$ sending $t$ to $\tau$. Then,
		$P$ is a nonzero prime ideal of $\F_q[t]$. Recall that  $\mathcal 
		O$ is the integral closure of $\F_q[t]$ in 
		$L=\F_q(t,\alpha_1,\ldots,\alpha_n)$. Let $\mathfrak P\lhd\mathcal O$ be a 
		prime ideal lying over $P$. Then, $\mathcal O/\mathfrak P$ is a 
		finite extension field of $\F_q[t]/P$ and thus can be embedded 
		in $\overline{\F_q}$ in a way compatible with 
		$\F_q[t]/P\xrightarrow{\sim}\F_q(\tau)$ (via $t\mapsto\tau$). This 
		embedding composed with the quotient map $\mathcal 
		O\mapsto\mathcal O/\mathfrak P$ gives an $\F_q$-algebra 
		homomorphism $\psi\colon \mathcal O\to\overline{\F_q}$ with $\psi(t)=\tau$, i.e.\ a specialization map for $f,\tau$.
		
		{\it (ii)} Obvious. Indeed, $\psi$ is an $\F_q$-algebra homomorphism with $\psi(t)=\tau$ and therefore $f_\tau=\prod_{\alpha\in Z(f)}(X-\psi(\alpha))$.
		
		{\it (iii)} This follows from \cite[Theorem 9.6]{Ros02} since $\Fr_{q^r}\in\Gal\left((\mathcal O/\mathfrak P))/(\F_q[t]/P)\right)$, where $\mathfrak P=\ker\psi$ and $P=\mathfrak P\cap\F_q[t]=\ker(\F_q[t]\to\F_q(\tau):t\mapsto\tau).$
		
		{\it (iv)} If $f_\tau$ is separable then $\#Z(f)=\#Z(f_\tau)=n$, so the surjective map $\psi:Z(f)\to Z(f_\tau)$ is also bijective.
		
		{\it (v)} This follows from \cite[Propositions 9.7, 9.10]{Ros02}.
	\end{proof}
	
	\begin{dfn} Let $f$ be as in Proposition~\ref{prop:spec}, $\tau\in\overline{\F_q}$ and assume $f_\tau$ is separable. The conjugacy class of $\sigma\in G_f$ satisfying the assertion of Proposition~\ref{prop:spec}(iii) is called the \emph{Frobenius class}\footnote{The term \emph{Artin symbol} is also commonly used.} of $f,\tau$ and denoted by $\Fr(f;\tau)$. It is well-defined, and is independent of the choice of specialization map $\psi$, by Proposition~\ref{prop:spec}(v).\end{dfn}
	
	\begin{lemma} \label{lem:chebotarev} With $f$ as above, the union of $\Fr(f;\tau)$ for all $\tau\in\overline{\F_q}$ such that $f_\tau$ is separable generates $G_f$.\end{lemma}
	
	\begin{proof} 
		This follows from the Chebotarev Density Theorem in function fields \cite[6.4.8]{FrJa08}.
	\end{proof}
	
	We now apply the above theory to the case of additive polynomials.
	
	\begin{prop}\label{prop: char poly} Let $f=X^{q^n}+a_{n-1}X^{q-1}+\ldots+a_0X,a_i\in\F_q[t],a_0\neq 0$ be a separable $q$-additive polynomial, $\tau\in\overline{\F_q}$. Let $\tilde f=X^n+a_{n-1}X^{n-1}+\ldots+a_0$ be the associated polynomial of $f$. 
		
		\begin{enumerate}
			\item [(i)] If $f_\tau$ is separable, then the determinant of (any element of) $\Fr(f;\tau)$ acting on $Z(f)$ is \\ $(-1)^{rn}N_{\F_{q^r}/\F_q}(a_0(\tau))$, where $r=[\F_q(\tau):\F_q]$.
			\item[(ii)] Assume that $\tau\in\F_q$ and write $\tilde f_\tau=X^kg,g(0)\neq 0$. There exists $\sigma\in G_f=\Gal(f/\F_q(t))$ such that the characteristic polynomial of $\sigma|_{Z(f)}$ is divisible by $g$.
		\end{enumerate}
		
	\end{prop}
	
	\begin{proof} (i) follows by combining Proposition~\ref{prop:spec}(iv-v) with Proposition~\ref{prop:frob_finite}(iii). For the second part let $\mathcal O$ be the integral closure of $\F_q[t]$ in $\F_q(t,Z(f))$ and $\psi:\mathcal O\to\overline{\F_q}$ a specialization map for $f,\tau$ (its existence is guaranteed by Proposition~\ref{prop:spec}(i)). By Proposition~\ref{prop:spec}(iii) there exists $\sigma\in G_f$ such that \begin{equation}\label{eq:frob_elt}\psi(\sigma(\alpha))=\psi(\alpha)^{q}\end{equation} for all $\alpha\in Z(f)$.
		
		By Proposition~\ref{prop:spec}(ii), $\lambda=\psi|_{Z(f)}:Z(f)\to Z(f_\tau)$ is surjective. It is also a linear map of $\F_q$-vector spaces because $\psi$ is an $\F_q$-algebra homomorphism.
		By (\ref{eq:frob_elt}) we have that $\sigma(\ker\lambda)=\ker\lambda$. Hence (once again using (\ref{eq:frob_elt})) we have the equality of characteristic polynomials 
		\begin{equation}\label{eq:char_pol1}P_{\sigma|_{Z(f)}}=P_{\sigma|_{\ker\lambda}}P_{\Fr_{q}}|_{Z(f_\tau)}\end{equation}
		
		Now writing $g=x^{n-k}+g_{n-k-1}x^{n-k-1}+\ldots+g_0,\,g_i\in\F_q,g_0=g(0)\neq 0$ and $h=x^{q^{n-k}}+g_{n-k-1}x^{q^{n-k-1}}+\ldots+g_0x$ we see that $f_\tau=h^{q^k}$ and $Z(f_\tau)=Z(h)$. The associated polynomial of $h$ is $\tilde h=g$ and by Proposition~\ref{prop:frob_finite}(ii) we have $P_{\Fr_q|_{Z(h)}}=g$. Plugging this into (\ref{eq:char_pol1}) we obtain the assertion of (ii).
	\end{proof}
	
	\subsection{A height inequality}
	\label{sec:height}
	
	For a field $F$, a non-archimedean absolute value $\|\cdot\|$ on $F$ and polynomial $f=\sum_{i=0}^na_iX^i$ we define its Gauss norm with respect to $\|\cdot\|$ to be $\|f\|=\max_{0\le i\le n}\|a_i\|$. The Gauss norm is multiplicative, i.e. $\|fg\|=\|f\|\|g\|$ (this is a generalized version of the Gauss lemma).
	
	\begin{lemma} Let $F$ be a field with $\chr F=p$ equipped with an absolute value $\|\cdot\|$ and let $f,g\in F[X]$ be monic separable additive polynomials such that $g \mid f$. Then
		$\|f\|\ge \|g\|^{\deg f/\deg g}$.
	\end{lemma}
	
	\begin{proof} The absolute value is non-archimedean since $F$ has 
		positive characteristic. By \cite[Theorem XII.4.1]{Lan02} it can be 
		extended to an absolute value on $\overline F$, which we also denote by 
		$\|\cdot\|$. Since $g \mid f$, $Z(g)\subset Z(f)$ is a linear 
		$\F_p$-subspace. By the separability assumption $\# Z(f)=\deg 
		f,\# Z(g)=\deg g$. Let $S\subset Z(f)$ be a system of 
		representatives of $Z(f)/Z(g)$. We have $\# S=\deg f/\deg g$. Using 
		the additivity of $g$,
		$$f=\prod_{\beta\in S}\prod_{\alpha\in Z(g)}(X-\alpha-\beta)=\prod_{\beta\in S}g(X-\beta)
		=\prod_{\beta\in S}(g(X)-g(\beta)).$$
		Since $g(0)=0$ we have $\|g(X)-g(\beta)\|\ge \|g\|$ and therefore
		$$
		\|f\|=\prod_{\beta\in S}\|g(X)-g(\beta)\|\ge \|g\|^{\#S}=\|g\|^{\deg f/\deg g},
		$$
		as required.\end{proof}
	
	\begin{corollary}\label{cor:e_bound} Let $k$ be a field with $\chr k=p$ and let $$f=X^{q^n}+a_{n-1}X^{q^{n-1}}+\ldots+a_0X,\quad g=X^{q^m}+b_{m-1}X^{q^{m-1}}+\ldots+b_0X\in k[t,X]$$ be $q$-additive polynomials with $a_i,b_i\in k[t]$, $a_0,b_0\neq 0$. Assume $g\mid f$. Then $\deg_tf\ge q^{n-m}\deg_tg$.\end{corollary}
	
	\begin{proof} Apply the previous lemma with $F=k(t)$ and $\|\cdot\|=q^{\deg(\cdot)}$.\end{proof}
	
	\section{Irreducibility}\label{sect:irred}
	
	The goal of the present section is to prove the following proposition which shows that for almost all $q$-additive polynomials in our (small box) model the Galois group $\Gal\left(f/\overline{\F_q}(t)\right)$ acts ``as irreducibly as possible'' (given $\con_tf$) on the space of roots.
	
	\begin{proposition}\label{prop:irred} Let $q$ be a fixed prime power, $d,\eta$ fixed natural numbers. Then the number of $q$-additive polynomials $f=X^{q^n}+a_{n-1}X^{q^{n-1}}+\ldots+a_0X$ with $a_i\in\F_q[t]_{\le d},a_0\neq 0$ such that $\deg\con_t(f)=\eta$ and
		which have a proper factor of the form $g=X^{q^m}+b_{m-1}X^{q^{m-1}}+\ldots+b_0X$ with $b_i\in\overline{\F_q}[t]$, $m<n$ and $\deg_tg>0$ is $O\left(q^{\left(\frac dq+1\right)n}\right)$ (the implied constant depends only on $q,d,\eta$).
	\end{proposition}
	
	The proof of the proposition occupies the rest of the section. Throughout the section we assume the setup of the proposition, in particular $q,d,\eta$ are fixed and $n$ is a varying parameter. All asymptotic notation will have implicit constant or rate of decay which may depend on $q,d,\eta$ but no other parameters. First we show that it is sufficient to prove the proposition with the condition $b_i\in\overline{\F_q}[t]$ replaced with $b_i\in\F_q[t]$. This is immediate from part (iii) of the following
	
	\begin{lemma}\label{lem:bounds_red_fq} Let $f,g,m,n$ be as in Proposition~\ref{prop:irred} (with $g\mid f$). Then
		\begin{enumerate}
			\item[(i)] $\deg_tg\le d/q$.
			\item[(ii)] $n-m\le\log_qd$.
			\item[(iii)] Assume further that $n>\eta+2\log_qd$. Then there exists a $q$-additive $g_2\in\F_q[t,X]$ with $g_2\mid f,\,\deg_Xg_2<n,\deg_tg_2>0$.
	\end{enumerate}\end{lemma}
	
	\begin{proof} By Corollary~\ref{cor:e_bound} we have $d\ge\deg_tf\ge q^{n-m}\deg_tg\ge q^{n-m}$. Since $\deg_tg\ge 1$ and $m<n$ this implies (i) and (ii).
		
		To prove (iii) we assume $n>\eta+2\log_qd$ and we may assume without loss of generality that $g$ has minimal degree in $X$ among monic $q$-additive proper divisors of $f$ with $\deg_tg>0$ (otherwise replace $g$ with a monic divisor of smaller degree). If $g\in\F_q[t,X]$ we are done. Otherwise assume $g\in\overline{\F_q}[t,X]\setminus\F_q[t,X]$ and let $g_1\neq g$ be a conjugate of $g$ over $\F_q$ (e.g. apply the $q$-Frobenius to each coefficient of $g\in\overline{\F_q}[t,X]$). Since $f\in\F_q[t,X]$ we have also $g_1\mid f$.
		
		Denote $u=\mathrm{gcd}(g,g_1)$ (which is in $\FF_q[t,X]$ by Gauss's lemma). By Lemma~\ref{lem:gcd}, $u$ is a $q$-additive divisor of $f$ and we have $\deg_Xu<\deg_X g$. By the minimality assumption on $\deg_Xg$ we must have $\deg_tu=0$, i.e. $u\in\overline{\F_q}[X]$ and therefore $u\mid \con_tf$ and $\deg_Xu\le\eta=\deg\con_tf$. We have $Z(g)\cap Z(g_1)=Z(u)$.
		Since $g,g_1\mid f$ we have $Z(f)\supset Z(g)+Z(g_1)$ and therefore 
		$$n=\dim Z(f)\ge \dim Z(g)+\dim Z(g_1)-\dim Z(u)=2m-\deg_Xu\ge 2m-\eta.$$
		Combining the inequalities $n\ge 2m-\eta$ and $n-m\le\log_qd$ we get $n\le\eta+2\log_qd$, contradicting the assumption of part (iii). This completes the proof.\end{proof} 
	
	It remains to prove Proposition~\ref{prop:irred} with $\overline{\F_q}$ replaced by $\F_q$ in the statement.
	For 
	$$
	g=X^{q^m}+b_{m-1}X^{q^{m-1}}+\ldots+b_0X,b_i\in\F_q[t],\qquad b_0\neq 0,
	$$ 
	denote
	$$\mathcal F_g=\left\{f=a_nX^{q^n}+a_{n-1}X^{q^{n-1}}+\ldots+a_0X,a_i\in\F_q[t]_{\le d}:g\mid f\right\}$$
	(the parameter $n$ will be implicit in the notation).
	Note that $\mathcal F_g$ is a linear $\F_q$-subspace of $\F_q[t,X]$. Denote also
	$\mathcal F_g'=\{f\in\mathcal F_g\mbox{ monic,}\deg_Xf=q^n\}$. It follows from Lemma~\ref{lem:bounds_red_fq}(i-ii) that $\mathcal F_g'$ is empty unless $n-\log_qd\le m\le n-1$ and $\deg b_i\le d/q$.
	
	Hence, the set of polynomials $f$ we want to bound is contained in $\bigcup_{g\in\mathcal G}\mathcal F_g$, where
	$$
	\mathcal G=\left\{g=X^{q^m}+b_{m-1}X^{q^{m-1}}+\ldots+b_0X\,\middle|\, n-\log_qd\le m\le n-1,\deg b_i\in\F_q[t]_{\le d/q}\right\},
	$$ 
	and it remains to show
	$\#\bigcup_{g\in\mathcal G}\mathcal F_g=O\left(q^{\left(\frac dq+1\right)n}\right)$.
	
	We have $\#\mathcal G\le q^{\left(\frac dq+1\right) n}$ since each coefficient $b_i\in\F_q[t]$ has degree $\le d/q$ and $0\le i\le n-1$.
	For a given $g\in\mathcal G$ with $\deg_X g=q^m$ consider the $\F_q$-linear map 
	$$\ogphi:\mathcal F_g\to\left(\F_q[t]_{\le d}\right)^{n-m+1}$$ given by 
	$\ogphi\left(a_nX^{q^n}+\ldots+a_0X\right)=(a_n,a_{n-1},\ldots,a_m)$. Note 
	that $\ogphi$ is injective because if $f=a_nX^{q^n}+\ldots+a_0X\in\mathcal 
	F_g$ with $a_n=a_{n-1}=\ldots=a_m=0$ then $f=0$, since $f$ is divisible by 
	$g$ which has degree $q^m$ in $X$, so we cannot have $\deg_X f<q^m$. Hence $\dim\mathcal 
	F_g\le (d+1)(n-m+1)$ and $\#\mathcal F_g\le q^{(d+1)(n-m+1)}=O(1)$ since 
	$n-m\le\log_qd=O(1)$ if $g\in\mathcal G$. From this we deduce
	$\#\bigcup_{g\in\mathcal G}\mathcal F_g=O(\#\mathcal G)=O\left(q^{\left(\frac dq+1\right)n}\right)$,
	which concludes the proof of Proposition~\ref{prop:irred}.
	
	\begin{corollary}\label{cor:irred}
		Let $q$ be a fixed prime power and $d,\eta$ fixed natural numbers. The number of additive polynomials $f=X^{q^n}+a_{n-1}X^{q^{n-1}}+\ldots+a_0X$ with $a_i\in\F_q[t]_{\le d},a_0\neq 0$ such that $\deg\con_t(f)=\eta$ and the quotient action of $G_f=\Gal(f/\F_q(t))$ on $Z(f)/Z(\con_t f)$ is not irreducible
		is $O\left(q^{\left(\frac dq+1\right)n}\right)$.
	\end{corollary}
	
	\begin{proof} 
		Note that since $\con_tf\in\F_q[X]$, the subspace $Z(\con_tf)\subset Z(f)$ is invariant under $G_f$ and the action of $G_f$ on $Z(f)/Z(\con_tf)$ is well-defined. The corollary now follows by combining Proposition~\ref{prop:irred} with Lemma~\ref{lem:factor_inv_subspace} to see that for all but $O\left(q^{\left(\frac dq+1\right)n}\right)$ choices of $f$ the space $Z(f)/Z(\con_tf)$ has no proper nontrivial $G_f$-invariant subspaces, since such a subspace lifts to a $G_f$-invariant subspace $Z(\con_tf)\subsetneq W\subsetneq Z(f)$ which corresponds by Lemma~\ref{lem:factor_inv_subspace} to a $q$-additive factor $g \mid f$ with $\deg_tg\ge 1$.
	\end{proof} 
	
	\section{Maximal subgroups of $\GL_n(q)$}
	\label{sect:subgroups}
	
	The maximal subgroups of $\GL_n(q)$ not containing $\SL_n(q)$ fall into nine Aschbacher classes $\mathcal{C}_1, ..., \mathcal{C}_8$ and $\mathcal{S}$, most of which can be described as stabilizers of certain structures relating to the natural module $V=\FF_q^n$. For precise descriptions of the classes, see \cite{aschbacher1984maximal} or \cite{kleidman1990subgroup}. In Section~\ref{sect:lowerright}, we will show that $G_f$ contains a large special linear group with high probability by showing that $G_f$ is unlikely to be contained in any $\mathcal C_i$ or $\CS$. The proof will rely on counting the polynomials appearing as characteristic polynomials of elements of groups in these classes.
	
	The most interesting classes for our purposes are \begin{align*}
		\mathcal{C}_1:& \mbox{ the maximal subgroups leaving invariant a subspace of } V,\\
		\mathcal{C}_2:& \mbox{ maximal subgroups of the form }\GL_{n/l}(q)\wr S_l \mbox{ for }l\mid n,\\
		\mathcal{C}_3:& \mbox{ maximal subgroups of the form }\GL_{n/b}(q^b).b \mbox{ for prime }b\mid n.
	\end{align*}
	
	%In Section~\ref{sect:irred} we dealt with reducible polynomials. Now we will deal with irreducible ones. In our asymptotic notation we fix the parameter in the subscript and allow the other to tend to infinity.
	
	\begin{prop}\label{prop:irreducible subgroups} Let $\CI= \bigcup_{i=2}^8\mathcal C_i \cup \CS$ (this is the set of all irreducible maximal subgroups of $\GL_n(q)$ not containing $\SL_n(q)$). The set of characteristic polynomials of elements in $\bigcup\CI$ has size $o\left(q^{n}\right)$ where the implied constant is an absolute constant.
	\end{prop}
	\begin{proof}
		We count the number of polynomials appearing as characteristic polynomials of elements in members of each class using results from \cite{fulman2012bounds, eberhard2023conjugacy}. The proposition then follows by a union bound. When we refer to the contribution of a class $\mathcal{C}$ of maximal subgroups of $\GL_n(q)$, we mean the number of characteristic polynomials of elements of groups in $\mathcal{C}$.
		
		The contribution of $\mathcal{C}:=\bigcup_{i=4}^8\mathcal C_i \cup \CS$ is at most the number of conjugacy classes of $\GL_n(q)$ which intersect a member of $\mathcal{C}$. By \cite[Lemmas~7.5-6, 9-12]{fulman2012bounds}, as noted in \cite[Theorem~4.8]{eberhard2023conjugacy}, this is $o(q^n)$.
		
		Class $\mathcal{C}_2$ consists of subgroups which preserve an additive decomposition $U \oplus V=\mathbb{F}_q^n$. Class $\mathcal{C}_3$ consists of subgroups preserving a field extension structure of prime degree. We claim we must only count semisimple conjugacy classes in these cases. The Jordan--Chevalley decomposition of a matrix $g \in \GL_n(q)$ is $g=us=su$ where $s$ is semisimple (i.e. diagonalizable over $\overline{\FF_q}$) and $u$ is unipotent (i.e. has all eigenvalues equal to $1$). Since conjugacy preserves semisimplicity, we may refer to conjugacy classes of $\GL_n(q)$ containing a semisimple element as semisimple conjugacy classes. If $M \in \mathcal{C}_2 \cup \mathcal{C}_3$ and $g\in M$ has Jordan--Chevalley decomposition $g=us=su$ then $s$ has the same characteristic polynomial as $g$. Moreover, any additive decomposition of $\mathbb{F}_q^n$ and any field extension of $\mathbb{F}_q$ which is preserved by $g$ is also preserved by $s$, since $s$ may be written as a polynomial in $g$. Therefore if $g$ is in a member of $\mathcal{C}_i$ for $i=2, 3$ then there is a semisimple element with the same characteristic polynomial in a member of $\mathcal{C}_i$. This implies that the contribution of $\mathcal{C}_2 \cup \mathcal{C}_3$ is the number of semisimple conjugacy classes of $\GL_n(q)$ which intersect a member of $\mathcal{C}_2 \cup \mathcal{C}_3$.
		
		Class $\mathcal{C}_2$ consists of subgroups conjugate to $M=\GL_{n/l}(q)\wr S_l$ for some $l>1$ dividing $n$. By \cite[Theorems~4.6]{eberhard2023conjugacy} the number of characteristic polynomials contributed by $M$ is $O(q^n n^{-\delta}(\log n)^{-1/2})$ where $\delta=1-(1+\log \log 2)/\log 2 \approx 0.086$ (note that in their notation $1-\delta_{\rm cc,ss}(G,M)$ is the proportion of characteristic polynomials in $\GL_n(q)$ which are contributed by a maximal subgroup $M$). Moreover it is well known (see, for example, \cite[Theorem 2.11]{MoVa06}) that the number of divisors of $n$ is $n^{o(1)}$, so summing over all $l$ dividing $n$ gives us that $\mathcal{C}_2$ contributes at most $O(q^n n^{o(1)-\delta}(\log n)^{-1/2})=o(q^n)$ characteristic polynomials.
		
		Class $\mathcal{C}_3$ consists of subgroups conjugate to $M=\GL_{n/b}(q^b).b$ for some prime $b\mid n$. By \cite[Theorem~4.7]{eberhard2023conjugacy} the contribution of $M$ is $O(q^n n^{-1/4}\log n)$, so summing over our at most $\log n$ prime divisors $b$ of $n$ gives us a contribution of $O(q^n n^{-1/4}(\log n)^2)=o(q^n)$.
		
		%By \cite[Lemma~7.5]{fulman2012bounds} there are at most $8n\log n + n \log \log q$ conjugacy classes of maximal subgroups in $\cup_{i=4}^8\mathcal{C}_i$ and by \cite[Lemma~7.6]{fulman2012bounds} for each $M \in \cup_{i=4}^8\mathcal{C}_i$ we have $k(M)\le Cq^{(n+1)/2}$ for some universal constant $C$. This gives us a contribution of at most $Cq^{(n+1)/2}(8n\log n + n \log \log q)=o(q^n)$ for classes $\mathcal{C}_4$, ..., $\mathcal{C}_8$.
		
		%Finally, Fulman and Guralnick's argument \cite[Proofs of Lemmas 7.9-7.12]{fulman2012bounds} gives us our result for class $\CS$.
	\end{proof}
	
	\section{Lower right block}\label{sect:lowerright}
	
	Throughout the section, we fix a prime power $q$ and a natural number $d$. We also use the notation $\tilde g$ for the associated polynomial of a $q$-additive $g\in\F_q[X]$, as defined by (\ref{eq:def tilde}). Let $f=X^{q^n}+a_{n-1}X^{q^{n-1}}+\ldots+a_0X\in\F_q[t,X],\,a_0\neq 0$ be a separable $q$-additive polynomial. By Corollary~\ref{cor:content additive}, $h=\con_tf$ is also additive. Recall that $G_f=\Gal(f/\F_q(t))$ acts $\F_q$-linearly on $Z(f)$, and clearly $Z(h)$ is an invariant subspace for this action (since $h\in\F_q[X]$). Denote by $\rho_f:G_f\to\GL(Z(f)/Z(h))$ the corresponding quotient representation. In the present subsection, we will show that for almost all $f$ in our model the image of $\rho_f$ is large, i.e. $\rho_f(G_f)\supset\SL(Z(f)/Z(h))$. In particular, this will imply the first statement in Theorem~\ref{thm:1}.
	
	\begin{prop}\label{prop:lower right corner} Let $q$ be a fixed prime power, $d$ a fixed natural number, $0\neq a_0(t)\in\F_q[t]_{\le d}$ a fixed polynomial and $h\in \F_q[X]$ a fixed separable $q$-additive polynomial. Then the number of polynomials of the form
		$$f=X^{q^n}+a_{n-1}X^{q^{n-1}}+\ldots+a_0X,\,a_i(t)\in\F_q[t]_{\le d}$$
		such that $\con_tf=h$ and $\rho_f(G_f)\not\supset\SL(Z(f)/Z(h))$ is $o(q^{(d+1)n})$ as $n\to\infty$.\end{prop}
	
	The proof of Proposition~\ref{prop:lower right corner} occupies the rest of the section. The strategy of the proof is as follows: first, Corollary~\ref{cor:irred} guarantees that almost surely $\rho_f(G_f)$ acts irreducibly on $Z(f)/Z(h)$. Then we specialize $f(t,X)$ with $t=0$ and using Proposition~\ref{prop: char poly}(ii) deduce that $\rho_f(G_f)$ contains an element with characteristic polynomial of the form $\frac {u(X)}{X^k}\tilde f(0,X)/\tilde h(X)$ where $k=\ord_X(f)$ and $\deg u=k$. We will show that the distribution of $\tilde f(0,X)/\tilde h$ is close to uniform in $\F_q[X]_{n-\deg \tilde h}^\monic$ (in a sense which will be made precise) as $f$ varies in our family, and the characteristic polynomials constructed in the above way will also have distribution close to uniform. Now Proposition~\ref{prop:irreducible subgroups} shows that there are only $o(q^n)$ possible characteristic polynomials of elements which lie in some subgroup of $\GL(Z(f)/Z(h))$ which does not contain $\SL(Z(f)/Z(h))$. This shows that almost surely $\rho_f(G_f)\supset\SL(Z(f)/Z(h))$.
	
	Now we carry out the strategy outlined above.  We fix a monic, separable, $q$-additive polynomial $h\in\F_q[X]$ and denote $\eta=\deg\tilde h$. We also fix $$0\neq a_0=\sum_{i=0}^dc_it^i\in\F_q[t]_{\le d}.$$ Throughout the rest of the section all asymptotic notation has implicit constant or rate of convergence depending on $q,d,h,a_0$.
	Denote
	$$\mathcal F=\mathcal F_{h,a_0}=\left\{f=X^{q^n}+a_{n-1}X^{q^{n-1}}+\ldots+a_0X\,\mid\,a_i\in\F_q[t]_{\le d}\,(1\le i\le n-1),\,\con_tf=h\right\}.$$
	We equip $\mathcal F$ with the uniform probability measure. The notation $\Prob(E(f))$ will always refer to the probability of the event $E(f)$ with respect to the above measure. Note that for any $f\in\mathcal F$ we have $h\mid f$ and therefore $h\mid f(0,X)$ and $\tilde h\mid\tilde f(0,X)$ (by Lemma~\ref{lem:fq additive gcd}).
	
	\begin{lem}\label{lem:family count}
		$$\#\mathcal F\sim \frac{1-q^{-d}}{1-q^{-d-1}}q^{(d+1)(n-\eta-1)}$$ as $q,d,\eta=\deg \tilde h$ are fixed and $n\to\infty$.
	\end{lem}
	\begin{proof} Any $f\in\mathcal F$ can be written uniquely as $f=A_d(X)t^d+A_{d-1}(X)t^{d-1}+\ldots+A_0(X)$ with $A_i\in\F_q[X]$ being $q$-additive, $\tilde A_0=\tilde f(0,X)$ monic of degree $n$ and $\tilde A_i$ of degree $\le n-1$ whenever $i>0$. The condition $\con_tf=\gcd(A_0,\ldots,A_d)=h$ is equivalent (via Lemma~\ref{lem:fq additive gcd}) to $\gcd(\tilde A_0,\ldots,\tilde A_d)=\tilde h$ and the condition on the $a_0$ coefficient is equivalent to $\tilde A_i(0)=c_i$. Writing $\tilde A_i=\tilde hB_i$ we see that $\mathcal F$ is in a bijective correspondence with the set
		$$\mathcal S= \left\{(B_0,\ldots,B_d)\in\left(\F_q[X]_{\le n-\eta-1}\right)^{d}\times\F_q[X]_{n-\eta}^\monic\,\mid\,B_i(0)=c_i/{\tilde h}(0),\,\gcd(B_0,\ldots,B_d)=1\right\}.$$
		A standard sieving argument (as in \cite{Nym72}) shows that  \begin{multline}\label{eq:fscount}\#\mathcal F=\#\mathcal S \sim C\cdot\#\left\{(B_0,\ldots,B_d)\in\left(\F_q[X]_{\le n-\eta-1}\right)^d\times\F_q[X]_{n-\eta}^\monic\,\mid\,B_i(0)=c_i/\tilde h(0)\right\}=\\=Cq^{(d+1)(n-\eta-1)},\end{multline} where
		\begin{equation}\label{eq:euler product}C=\prod_{P\in\F_q[X]\atop{\mathrm{prime}\atop{P\neq X}}}\left(1-q^{-(d+1)\deg P}\right)=\frac{1-q^{-d}}{1-q^{-d-1}}.\end{equation}
		Here the prime $P=X$ is omitted from the product because the condition $a_0\neq 0$ implies that not all $c_i$ are 0 and therefore $X$ cannot be a common divisor of $B_0,\ldots,B_d$. The last equality in (\ref{eq:euler product}) follows from \cite[\S 2, equations (1),(2)]{Ros02}). The assertion of the lemma follows from (\ref{eq:fscount}) and (\ref{eq:euler product}).
	\end{proof}

	\begin{lem}\label{lem:uniform}Let $\ogphi:\mathcal F\to\F_q[X]_n^{\mathrm{monic}}$ be a function such that for any $f\in \mathcal F$ we have $\ogphi(f)=\tilde f(0,X)X^{-k}u(X)$ for some $k\in\ZZ_{\ge 0}$ and $u\in\F_q[X]_{k}^\mathrm{monic}$ depending on $f$. Then for any $\mathcal T\subset\F_q[X]_n^{\mathrm{monic}}$ we have
		$$\Prob(\ogphi(f)\in\mathcal T) =O\left(\left(q^{-n}\#\mathcal T\right)^{1/2}\right).$$ 
	\end{lem} 
	
	\begin{proof} Consider the map $\psi:\mathcal F\to\F_q[X]_n^{\monic}$ given by $\psi(f)=\tilde f(0,X)$. We have $|\psi^{-1}(v)|\ll q^{dn}$ (this is seen by writing $\tilde f=\tilde h\sum_{i=0}^dB_i(X)t^i$ with $B_i\in\F_q[X]_{\le n-\eta}$ and $B_0=v/\tilde h$) and by Lemma~\ref{lem:family count} we also have $\#\mathcal F\gg q^nq^{dn}$. Hence for any $\mathcal U\subset\F_q[X]_n^{\mathrm{monic}}$ we have \begin{equation}\label{eq:psi uniform}\Prob\left(\tilde f(0,X)\in \mathcal U\right)\ll q^{-n}\cdot \#\mathcal U.\end{equation}
		
		Now let $l$ be a natural number. From (\ref{eq:psi uniform})  we have $\Prob\left(X^l\mid \tilde f(0,X)\right)\ll q^{-l}$. On the other hand if $w\in\F_q[X]_n^\monic$ then there are at most $q^{l-1}$ polynomials $v\in\F_q[X]_n^\monic$ such that $vX^{-k}u=w$ for some $k,u$ and $X^l\nmid v$ (since $\deg u=k$ and $X^{k-l+1}|u$). Using (\ref{eq:psi uniform}) again we see that for $\mathcal T\subset\F_q[X]_n^\monic$ we have
		$$\Prob(\ogphi(f)\in\mathcal T)\le\Prob\left(X^l\mid \tilde f(0,X)\right)+\Prob\left(X^l\nmid \tilde f(0,X)\mbox{ and }\ogphi(f)\in\mathcal T\right)\ll q^{-l}+q^{l-n}\#\mathcal T.$$
		Now choosing $l=\big\lceil -\log_q\left(q^{-n}\#\mathcal T\right)^{1/2}\big\rceil$ we obtain the assertion of the lemma.
	\end{proof}
	
	We are now ready to prove Proposition~\ref{prop:lower right corner}. By Corollary~\ref{cor:irred} combined with Lemma~\ref{lem:family count} we have $\Prob\left(\rho_f(G_f)\mbox{ reducible}\right)=o(1)$. Next, by Proposition~\ref{prop: char poly}(ii) (taking the specialization $\tau=0$) there is an element $\sigma\in G_f$ such that the characteristic polynomial of $\sigma$ acting on $Z(f)$ is of the form $X^{-k}u\tilde f(0,X)$ for some $u\in\F_q[X]_k^\monic$. Denote this characteristic polynomial by $\ogphi(f)$. Note that the map $\ogphi:\mathcal F\mapsto\F_q[X]_n^\monic$ satisfies the assumption of Lemma~\ref{lem:uniform}.
	
	Denote by $\CI$ the set of irreducible subgroups of $\GL(Z(f)/Z(h))$ not containing $\SL(Z(f)/Z(h))$ and let $\mathcal T\subset\F_q[X]$ be the set of characteristic polynomials of elements $A\in\GL(Z(f))$ such that $AZ(h)=Z(h)$ and the quotient map $A|_{Z(f)/Z(h)}\in\bigcup\CI$ lies in some subgroup of $\GL(Z(f)/Z(h))$ not containing $\SL(Z(f)/Z(h))$. By Proposition~\ref{prop:irreducible subgroups} we have $\#\mathcal T=o(q^n)$ (note that there are a priori $O(1)$ possibilities for $A|_{Z(h)}$; moreover the way $\sigma$ is actually constructed shows that $\sigma_{Z(h)}=\Fr_q$ is uniquely determined).
	
	Now using Lemma~\ref{lem:uniform} we obtain
	\begin{multline*}\Prob\left(\rho_f(G_f)\not\supset\SL(Z(f)/Z(h))\right)=\\=\Prob\left(\rho_f(G_f)\mbox{ irreducible and }\rho_f(G_f)\not\supset\SL(Z(f)/Z(h))\right)+o(1)\le\\
		\le\Prob(\ogphi(f)\in\mathcal T)+o(1)=o(1),\end{multline*}
	which completes the proof of Proposition~\ref{prop:lower right corner}.
	
	\section{Determinants}
	\label{sec:det}
	
	Let $q$ be a fixed prime power. Let $$f=X^{q^n}+a_{n-1}X^{q^{n-1}}+\ldots+a_0X,\,a_i\in\F_q[t],\,a_0\neq 0$$ be a separable $q$-additive polynomial and denote
	$$h=\con_tf=X^{q^\eta}+h_{\eta-1}X^{q^{\eta-1}}+\ldots+h_0X\in\F_q[x]$$
	($h$ is $q$-additive by Corollary~\ref{cor:content additive}).
	Denote $G=G_f=\Gal(f/\F_q(t))$. Since $h\mid f$ we have $Z(h)\subset Z(f)$ and since $h\in\F_q[X]$ the subspace $Z(h)$ is invariant for the action of $G$ on $Z(f)$ (Lemma~\ref{lem:factor_inv_subspace}). Consequently, the $\F_q$-linear action of $G$ on $Z(f)/Z(h)$ is well-defined, and we denote by $\rho:G\to\GL(Z(f)/Z(h))$ the corresponding quotient representation. We consider the homomorphism 
	$\delta:G\to\GL(Z(h))\times\F_q^\times$ given by $$\delta(\sigma)=\left( \sigma|_{Z(h)},\det\rho(\sigma)\right).$$ In the present section we want to describe the image of $\delta$.
	
	\begin{prop}\label{prop:det}With notation as above write $a_0=cu^k,c\in\F_q^\times$ and $u\in\F_q[t]\setminus\F_q$ monic with $k\ge 0$ maximal (such that a representation of this form exists). Then
		\begin{equation}\label{eq:image_delta}\delta(G_f)=\left\{(\phi^r,s)\,\middle|\,r\in\ZZ,s\in\left((-1)^{n-\eta}\frac{c}{h_0}\right)^r\F_q^{\times k}\right\},\end{equation}
		where $\phi=\Fr_q|_{Z(h)}$.
	\end{prop}
	
	We first prove an auxiliary result on finite fields.
	
	\begin{lemma}\label{lem:weil}Let $u\in\F_q[t]$ be a nonconstant monic polynomial which is not an $l$-th power of another polynomial for any $l\ge 2$. Then for $r\ge r_0(u)$ sufficiently large and any $b\in\F_q^\times$ there exists $\tau\in \overline{\FF_q}$ with $[\F_q(\tau):\F_q]=r$ such that $N_{\F_{q^r}/\F_q}(u(\tau))=b$ (in particular, $u(\tau)\neq 0)$.
	\end{lemma}
	
	\begin{proof}
		This is a standard consequence of the Weil bound on exponential sums (in fact, even much weaker bounds suffice). We give the details for completeness.
		
		Let $b\in\F_q^\times$ and assume that for some $r$ and for any $\tau$ with $[\F_q(\tau):\F_q]=r$, $u(\tau)\neq 0$ we have $N_{\F_{q^r}/\F_q}(u(\tau))\neq b$. A trivial counting argument gives that 
		\begin{equation}\label{eq0}
			\#\{\tau\in\F_{q^r} : \F_q(\tau)\subsetneq \F_{q^r} \mbox{ or } u(\tau)=0\} \leq 2q^{r/2}+\deg u.
		\end{equation}

		Denote by $\chi_1,\ldots,\chi_{q-2}$ the set of nontrivial multiplicative characters of $\F_q^\times$. By the second orthogonality relation \cite[Proposition 4.2(2)]{Ros02} we have
		\begin{equation}\label{eq1}
			\sum_{i=1}^{q-2}\overline{\chi_i(b)}\chi_i(N_{\F_{q^r}/\F_q}(u(\tau)))=-1
		\end{equation}
		whenever $\F_{q^r}=\F_q(\tau),u(\tau)\neq 0$ and trivially 
		\begin{equation}\label{eq2}
			\left|\sum_{i=1}^{q-2}\overline{\chi_i(b)}\chi_i(N_{\F_{q^r}/\F_q}(u(\tau)))\right|\le q-1
		\end{equation}
		if either $u(\tau)=0$ or $\FF_q(\tau)\subsetneq \FF_{q^r}$. Applying \eqref{eq1} and \eqref{eq2} and then \eqref{eq0}, we obtain
		\begin{equation}\label{eq:weil1}
			\begin{split}
				\left|\sum_{i=1}^{q-2}\overline{\chi_i(b)}\sum_{\tau\in\F_{q^r}}\chi_i\left(N_{\F_{q^r}/\F_q}(u(\tau))\right) + q^r\right|&\leq \sum_{
					\substack{\tau\in \FF_{q^r}\\ u(\tau)=0 \mbox{\tiny \ or } \FF_q(\tau)\subsetneq \FF_{q^r}}
				}
				\left|\sum_{i=1}^{q-2}\overline{\chi_i(b)}\chi_i(N_{\F_{q^r}/\F_q}(u(\tau)))+1\right| 
				\\ & \le q(2q^{r/2}+\deg u).
			\end{split}
		\end{equation}
		
		By the Weil bound (see e.g. \cite[Theorem 1]{CaMo00}) we have
		$$\left|\sum_{\tau\in\F_{q^r}}\chi_i\left(N_{\F_{q^r}/\F_q}(u(\tau))\right)\right|\le (\deg u-1)q^{r/2}.$$ Using this and the triangle inequality, from  \eqref{eq:weil1},  we deduce that 
		\begin{align*}
			q(2q^{r/2}+\deg u) \geq q^r -\sum_{i=1}^{q-2} \left|\sum_{\tau\in\F_{q^r}}\chi_i\left(N_{\F_{q^r}/\F_q}(u(\tau))\right)\right|\geq q^r - (q-2)(\deg u-1)q^{r/2},
		\end{align*}
		which is impossible for $r$ sufficiently large (with $q,\deg u$ fixed).
	\end{proof}
	
	\begin{proof}[Proof of Proposition~\ref{prop:det}]
		Now we assume the setup of Proposition~\ref{prop:det} and in particular write $a_0=cu^k$ with $c\in\F_q^\times$, $u\in\F_q[t]$ monic and nonconstant such that $u$ is not an $l\ge 2$-th power of another polynomial. We also denote $\phi=\Fr_q|_{Z(h)}$. 
		
		Recall that by Lemma~\ref{lem:chebotarev} the group $G_f$ is generated by the union of the Frobenius classes $\Fr(f;\tau)$ for $\tau\in\overline{\F_q},\,a_0(\tau)\neq 0$.
		Let $\tau\in\overline{\F_q}$ such that $a_0(\tau)\neq 0$ and $r=[\F_q(\tau):\F_q]$. Then for 
		$\sigma\in\Fr(f;\tau)$ we have $\sigma|_{Z(h)}=\phi^r$ (since $h\in\F_q[t]$ we may identify 
		$Z(h)$ with $Z(h_\tau)$ and each specialization map for $f,\tau$ respects this 
		identification). By Proposition~\ref{prop: char poly} we also have 
		$$\det\sigma|_{Z(f)}=(-1)^{rn}N_{\F_{q^r}/\F_q}(a_0(\tau)),\quad
		\det\sigma|_{Z(h)}=(\det\phi)^r=\left((-1)^\eta h_0\right)^r$$ and consequently (recall that 
		$\rho(\sigma)$ is the quotient of $\sigma|_{Z(f)}$ by $\sigma|_{Z(h)}$) 
		$$\delta(\sigma)=(\phi^r,\det\rho(\sigma))=\left(\phi^r,(-1)^{r(n-\eta)}\frac{N_{\F_{q^r}/\F
				_q}(a_0(\tau))}{h_0^r}\right)=\left(\phi^r,\left((-1)^{n-\eta}\frac c{h_0}\right)^rN_{\F_{q^r}/\F
			_q}(u(\tau))^k\right).$$
		By Lemma~\ref{lem:weil} these elements generate the RHS of (\ref{eq:image_delta}) (note that the condition $u(\tau)\neq 0$ is equivalent to $a_0(\tau)\neq 0$) which concludes the proof of Proposition~\ref{prop:det}.
	\end{proof}
	
	\section{Upper right block}
	\label{sec:upper right block}
	The next proposition gives sufficient conditions for the Galois $G_f$ to contain a large unipotent block subgroup. 
	\begin{prop}\label{prop:upper right corner}
		Let $f=X^{q^n}+a_{n-1}X^{q^{n-1}}+\ldots+a_0X,a_i\in\F_q[t],a_0\neq 0$ be a $q$-additive polynomial, $h=\con_tf$ with $\log_q \deg h=\eta$. Assume that 
		\begin{enumerate}
			\item[(a)]$n\ge\max(\eta+4,2\eta+1)$.
			\item[(b)] $f$ has no proper $q$-additive divisor $g\in\overline{\F_q}[t][X]$ such that $\deg_tg>0$.
			\item[(c)] The image of the quotient representation $G_f\to\GL(Z(f)/Z(h))$ contains $\SL(Z(f)/Z(h))$.
		\end{enumerate}
		Let $e_1,\ldots,e_n$ be a basis of $Z(f)$ such that $e_1,\ldots,e_\eta$ is a basis of $Z(h)$. Then with respect to this basis $G_f=\Gal(f/\F_q(t))$ contains all matrices of the form $$\tmatrix {I_\eta}\star{I_{n-\eta}}.$$ 
	\end{prop}
	
	We first prove auxiliary results and then prove the proposition.
	\subsection{Upper triangular block matrix groups}
	
	Let $F$ be a field, $m,n\ge 0$ integers. Denote
	\begin{equation}\label{eq:def tmn}T_{m,n}=\left\{\tmatrix {I_m}XA:A\in\GL_n(F),X\in M_{m\times n}(F)\right\}\leqslant\GL_{m+n}(F)\end{equation}
	and let $\rho:T_{m,n}\to\GL_n(F)$ be the projection homomorphism to the lower right $n$-block.
	
	Let $e_1,\ldots,e_{m+n}$ be the standard basis of $F^{m+n}$ and $W=\langle e_1,\ldots,e_m\rangle$. Note that $W$ is invariant for the action of $T_{m,n}$. 
	
	\begin{lem}\label{lem:higman} Assume $n\ge 4$. Let $\pi:T_{1,n}\to\GL_n(F)$ be the projection to the lower right $n\times n$ block and let $\Gamma\leqslant T_{1,n}$ be a subgroup such that
		\begin{enumerate}
			\item[(a)]$\pi$ is injective on $\Gamma$.
			\item[(b)]$\pi(\Gamma)\supset\SL_n(F)$.
		\end{enumerate}
		Then there exists $U\in T_{1,n}$ such that
		$$U^{-1}\Gamma U\subset\left\{\tmatrix 1{}A:A\in\GL_n(F)\right\}.$$
	\end{lem}
	
	\begin{proof}
		By assumption (a) there is a map $\delta:\pi(\Gamma)\to M_{1\times n}(F)$ such that for each $A\in\pi(\Gamma)$ there exists a unique $a\in\Gamma$ of the form $a=\tmatrix 1{\delta(A)}A$. Since $\pi:\Gamma\to\GL_n(F)$ is an injective homomorphism we must have
		$$\tmatrix 1{\delta(AB)}{AB}=\tmatrix 1{\delta(A)}A\tmatrix 1{\delta(B)}B=
		\tmatrix 1{\delta(B)+\delta(A)B}{AB}$$
		and therefore $\delta(AB)=\delta(B)+\delta(A)B$. Hence $\delta$ is a 1-cocycle of $\pi(\Gamma)$ with coefficients in $M_{1\times n}(F)$ (with respect to the standard right action of $\pi(\Gamma)\leqslant\GL_n(F)$ on $M_{1\times n}(F)$).
		
		It is shown in \cite[Theorem 2.2]{Pol71} that assuming $n\ge 4$, we have $H^1(\GL_n(F),M_{1\times n}(F))=0$ (with respect to the standard right action) and the proof in \cite{Pol71} works verbatim to show $H^1(G,M_{1\times n}(F))=0$ for any group 
		$\SL_n(F)\leqslant G\leqslant\GL_n(F)$. Hence $\delta$ above is a 1-coboundary of $\pi(\Gamma)$, i.e. 
		it has the form $\delta(A)=u-uA$ where $u\in M_{1\times n}(F)$ 
		is a fixed row vector.
		
		Now taking $U=\tmatrix 1{-u}{I_n}$ we obtain for a general element $a=\tmatrix 1{u-uA}A$ of $\Gamma$,
		$$U^{-1}aU=\tmatrix 1{u}{I_n}\tmatrix 1{u-uA}A\tmatrix 1{-u}{I_n}=\tmatrix 1{}A$$
		and $U^{-1}\Gamma U$ has the required form.
	\end{proof}
	
	\begin{lem}\label{lem:higman cor}Assume $n\ge \max(4,m+1)$. Let $\Gamma\leqslant T_{m,n}$ be a subgroup such that 
		\begin{enumerate}\item[(a)]$\rho(\Gamma)\supset\SL_n(F)$.
			\item[(b)] Any proper $\Gamma$-invariant subspace of $F^{m+n}$ is contained in $W=\langle e_1,\ldots,e_m\rangle$.
		\end{enumerate}
		Then $\Gamma\supset\ker\rho$.\end{lem}
	
	\begin{proof}
		
		The map $X\mapsto\tmatrix {I_m}X{I_n}$ gives an isomorphism of abelian groups $M_{m\times n}(F)\xrightarrow{\sim}\ker(\rho)$. For 
		$$\xi=\tmatrix {I_m}X{I_n}\in\ker(\rho)\cap\Gamma,\quad a=\tmatrix {I_m}YA\in\Gamma$$ we have $a^{-1}\xi a=\tmatrix {I_m}{XA}{I_n}$. Hence by assumption (a) the (additive) group $$K=\left\{X\in M_{m\times n}(F):\tmatrix {I_m}X{I_n}\in\ker(\rho)\cap\Gamma\right\}\leqslant M_{m\times n}(F)$$ is invariant under right multiplication by all $A\in\SL_n(F)$. Since we assume $n>m$ it is in fact invariant under right multiplication by $A\in\GL_n(F)$: indeed, if $X \in K$ then the row space of $X$ has dimension at most $m<n$, so for every $A \in \GL_n(F)$ there is a matrix $B\in \GL_n(F)$ such that $\det(B)=\det(A)^{-1}$ and $XB=B$. Since $\det(BA)=1$, it follows that $XA=XBA\in K$.
		
		Hence $K$ is a right $\GL_n(F)$-submodule of $M_{m\times n}(F)$ and it is well known\footnote{This follows for example from the explicit Morita equivalence of $M_n(F)$ and $F$, see \cite[\S 3.12]{Jac89}.}
		that it can be brought into the standard form
		\begin{equation}\label{eq:K_std}bK=\left\{\left[\begin{array}{c}Y\\0_{k\times n}\end{array}\right]:Y\in M_{(m-k)\times n}\right\}\end{equation}
		after applying a $\GL_n(F)$-automorphism to $M_{m\times n}(F)$, which is the same as multiplying by some $b\in\GL_m(F)$ on the left.
		
		Note that $$\tmatrix b{}{I_n}\tmatrix {I_m}X{I_n}\tmatrix 
		b{}{I_n}^{-1}=\tmatrix {I_m}{bX}{I_n},$$ hence taking $\beta=\stmatrix
		b{}{I_n}$ we have that $\left\{X:\stmatrix {I_m}X{I_n}\in\ker(\rho)\cap 
		\beta\Gamma\beta^{-1}\right\}=bK$ has the form 
		(\ref{eq:K_std}). Since $\beta\Gamma\beta^{-1}$ also satisfies the 
		conditions (a) and (b) in the statement of the lemma and 
		$\Gamma\supset\ker\rho$ if and only if $\beta\Gamma\beta^{-1}\supset\ker\rho$, we may 
		assume without loss of generality that $b=I_m$ and 
		\begin{equation}\label{eq:K_std2}K=\left\{\left[\begin{array}{c}Y\\0_{k\times n}\end{array}\right]:Y\in M_{(m-k)\times n}\right\}.\end{equation} We need to show that $K=M_{m\times n}(F)$, i.e. that $k=0$.
		
		Assume by way of contradiction that $k\ge 1$. Consider the homomorphisms $\rho_1:T_{m,n}\to T_{1,n}$, $\pi:T_{1,n}\to \GL_n(F)$ acting as follows:
		$$\rho_1:\left[\begin{array}{ccc}I_{m-1}&&Y\\&1&v\\&&A\end{array}\right]\mapsto \tmatrix 1vA,\quad\pi:\tmatrix 1vA\mapsto A.$$ Clearly $\rho=\pi\circ\rho_1$. It follows from (\ref{eq:K_std2}) that $\ker\rho\cap\Gamma=\ker \rho_1\cap\Gamma$ and therefore $\pi$ is injective on $\rho_1(\Gamma)$ and $\pi(\rho_1(\Gamma))=\rho(\Gamma)\supset \SL_n(F)$. The conditions of Lemma~\ref{lem:higman} are satisfied by $\rho_1(\Gamma)$ and therefore its conclusion holds, i.e. there exists $U\in T_{1,n}$ such that 
		$$U^{-1}\rho_1(\Gamma)U\subset\left\{\tmatrix 1{}A:A\in\GL_n(F)\right\}.$$
		
		Taking $C=\left[\begin{smallmatrix} I_{m-1}\\ &U\end{smallmatrix}\right]$ we have
		$$C^{-1}\Gamma C\subset\left\{\left[\begin{array}{ccc}I_{m-1}&&Y\\&{1}&\\&&A\end{array}\right]:Y\in M_{(m-1)\times n}(F),A\in\GL_n(F)\right\}$$ and $C^{-1}\Gamma C$ has an invariant proper subspace $V=\langle e_1,\ldots,e_{m-1},e_{m+1},\ldots,e_{m+n}\rangle$ not contained in $W$. Since $C$ preserves $W$, the proper subspace $CV$ is invariant for $\Gamma$ and not contained in $W$, contradicting assumption (b).
	\end{proof}
	
	\subsection{Proof of Proposition~\ref{prop:upper right corner}}
	
	In the present subsection we assume the setup of Proposition~\ref{prop:upper right corner}, including assumptions (a)-(c). In particular $e_1,\ldots,e_n$ is a basis of $Z(f)$ such that $e_1,\ldots,e_\eta$ is a basis of $Z(h)$. Using this basis we identify $\GL(Z(f))$ with $\GL_n(q)$. 
	
	Denote $\F_q(Z(h))=\F_{q^\nu}$ (recall that $h\in\F_q[X]$) and observe that under the above identification of linear operators with matrices we have
	$$\Gamma:=\Gal(f/\F_{q^\nu}(t))=\{\sigma\in G_f:\sigma|_{Z(h)}=\mathrm{id}\}=T_{\eta,n-\eta}\cap G_f$$
	using the notation (\ref{eq:def tmn}) with $F=\F_q$ and recalling that 
	$Z(h)$ is invariant under $G_f$. We will now check that the conditions of Lemma~\ref{lem:higman cor} hold for $\Gamma$, with $m,n$ replaced by $\eta,n-\eta$ and $W=Z(h)$.
	
	First of all assumption (a) in Proposition~\ref{prop:upper right corner} 
	implies that $n-\eta\ge\max(4,\eta+1)$. Next, denoting by 
	$\rho:T_{\eta,n-\eta}\to\GL_{n-\eta}(q)$ the projection to the lower 
	right block we see by assumption (c) in Proposition~\ref{prop:upper 
		right corner} that $\rho(G_f)\supset\SL_{n-\eta}(q)$. Since $\SL_{n-\eta}(q)$ is perfect ($n-\eta\ge 4$) and $[G_f,G_f]\subset\Gamma$ (because $\Gal(\F_{q^\nu}/\F_q)$ is cyclic) we have $\rho(\Gamma)\supset[\rho(G_f),\rho(G_f)]\supset\SL_{n-\eta}(q)$, verifying assumption (a) of Lemma~\ref{lem:higman cor} in our case. 
	
	Now let 
	$R\subsetneq Z(f)$ be a subspace invariant under $\Gamma=G_f\cap 
	T_{\eta,n-\eta}$. Then $g=\prod_{\alpha\in 
		R}(X-\alpha)\in\F_{q^\nu}[t,X]$ (since its coefficients are invariant 
	under $\Gamma=\Gal(f/\F_{q^\nu}(t))$) is a proper $q$-additive divisor of $f$ (by Proposition~\ref{prop:basic}(iii)), hence 
	by assumption (b) of Proposition~\ref{prop:upper right corner} we have 
	$\deg_tg=0$, i.e. $g\in\F_{q^\nu}[X]$ and therefore $g\mid \con_tf=h$ and 
	$R\subset Z(h)=W$, verifying assumption (b) of Lemma~\ref{lem:higman cor} in our case.
	
	We verified all the conditions of Lemma~\ref{lem:higman cor} and hence its conclusion holds, namely
	$$G_f\supset\Gamma\supset\left\{\tmatrix {I_\eta}X{I_{n-\eta}}:X\in M_{\eta\times(n-\eta)}(q)\right\},$$ which is exactly the assertion of Proposition~\ref{prop:upper right corner}.

	\section{Proof of main theorems}
	\label{sec:assembly}
	
	We devote the present section to the proof of Theorem~\ref{thm:2} which implies Theorem~\ref{thm:1}. At the end of the section we will also give the derivation of Corollary \ref{cor:irreduc} from Theorem \ref{thm:2}. First we recap the setup of Theorem~\ref{thm:2}, and then we assemble the partial information on $G_f$ provided by Propositions~\ref{prop:lower right corner}, \ref{prop:det} and \ref{prop:upper right corner} to prove it.
	
	Let $q$ be a fixed prime power, $d$ a fixed natural number. Fix a monic $h=X^{q^{\eta}}+h_{\eta-1}X^{q^{\eta-1}}+\ldots+h_0X\in\F_q[X],\,h_0\neq 0$ (so $h$ is separable). Finally fix $0\neq a_0\in\F_q[x]_{\le d}$ and write $a_0=cu^k$ with $u\in\F_q[t]\setminus\F_q$ monic, $c\in\F_q^\times$ and $k$ maximal among such representations ($k=0$ iff $a_0\in\F_q$). Henceforth all asymptotic notation will have an implicit constant or rate of decay which may depend on $q,d,h,a_0$ (or equivalently $q,d,\eta$, since if we fix these, there are only finitely many possibilities for $h,a_0$) as $n\to\infty$, where $n$ is a variable parameter.
	
	Denote
	\begin{equation}\label{eq:calf} \mathcal F=\left\{f=X^{q^n}+a_{n-1}X^{q^{n-1}}+\ldots+a_1X^q+a_0X\,\mid\,a_1,\ldots,a_{n-1}\in\F_q[t]_{\le d},\,\con_tf=h\right\}.\end{equation}
	Theorem~\ref{thm:2} will follow once we show that, as $n\to\infty$, for a $1-o(1)$ proportion of $f\in\mathcal F$ we have 
	\begin{multline}\label{eq:gamma_nhck}G_f=\Gamma_{n,h,c,k}:=\\ \left\{\tmatrix{D^r}YA:r\in\ZZ,
		\det(A)\in\left((-1)^{n-\eta}\frac c{h_0}\right)^r\F_q^{\times k},Y\in M_{\eta\times(n-\eta)}(q)\right\}\leqslant\GL_n(q),\end{multline}
	with respect to a suitable basis of $Z(f)$, with $D$ being the companion matrix of $\tilde h$. 
	
	%Note that while here we are fixing $a_0$ and in the statement of Theorem \ref{thm:2} only $c,k$ are fixed\todo{We are fixing $a_0$ in Th2, just in the conditional probability statement}, this does not make a difference because the RHS of (\ref{eq:gamma_nhck}) depends only on $c,k$, so we may subdivide by the value of $a_0$.
	
	By Proposition~\ref{prop:frob_finite}(ii) we may choose a basis $e_1,\ldots,e_\eta$ of $Z(h)$ such that $\phi=\Fr_q|_{Z(h)}$ has the matrix $D$. For any $f\in\mathcal F$ we choose an arbitrary basis $e_1,\ldots,e_n$ of $Z(f)$ that extends the above basis of $Z(h)$. All matrices of elements $\sigma\in G_f$ will be with respect to this basis.
	
	Recall that for any $f\in\mathcal F$ we have (with respect to a basis as above)
	$$G_f\leqslant T:=\left\{\tmatrix BYA:B\in\GL_\eta(q),Y\in M_{\eta\times(n-\eta)}(q),A\in\GL_{n-\eta}(q)\right\}.$$ Denote by 
	$\pi,\beta,\rho$ the projections from $T$ to the upper left $\eta\times\eta$ block, the upper right $\eta\times(n-\eta)$ block and the lower right $(n-\eta)\times(n-\eta)$ blocks respectively.
	
	First, by Proposition~\ref{prop:lower right corner}, for a $1-o(1)$ proportion of $f\in\mathcal F$ we have $\rho(G_f)\supset\SL_{n-\eta}(q)$. Next, by Proposition~\ref{prop:det} we see that $G_f$ is contained in $\Gamma_{n,h,c,k}$ and for a $1-o(1)$ proportion of $f\in\mathcal F$ we have
	\begin{equation}\label{eq:pirho}(\pi\times\rho)(G_f)=\left\{(D^r,A):r\in\ZZ,\det(A)\in\left((-1)^{n-\eta}\frac c{h_0}\right)^r\F_q^{\times k}\right\}\end{equation}
	(recall that $\phi=\Fr_q|_{Z(h)}$ has the matrix $D$ with respect to the basis $e_1,\ldots,e_\eta$).
	
	Next, by Proposition~\ref{prop:irred} for a $1-o(1)$ proportion of $f\in\mathcal F$ condition (b) of Proposition~\ref{prop:upper right corner} is satisfied and by the first observation in the preceding paragraph so is condition (c). Hence for $n$ large enough the assertion of Proposition~\ref{prop:upper right corner} applies and we see that for a $1-o(1)$ proportion of $f\in\mathcal F$ we have
	\begin{equation}\label{eq:ulc_Gf}\ker(\pi\times\rho)=\left\{\tmatrix{I_\eta}Y{I_{n-\eta}}:Y\in M_{\eta\times(n-\eta)}(q)\right\}\subset G_f.\end{equation} 
	
	By (\ref{eq:pirho}),(\ref{eq:ulc_Gf}) and the definition of $\Gamma_{n,h,c,k}$ we see that $G_f=\Gamma_{n,h,c,k}$ for a $1-o(1)$ proportion of $f\in\mathcal F$. This concludes the proof. \qed
	
	\begin{proof}[Proof of Corollary \ref{cor:irreduc}]
		First assume $f\in \mathcal F=\mathcal F_{n,h,c,k}$ as defined in (\ref{eq:calf}). Since $\Gamma_{n,h,c,k}$ is transitive on $Z(f)\setminus Z(h)=Z(f/\con_tf)$, by Theorem 2 asymptotically almost surely $G_f=\Gamma_{n,h,c,k}$ is transitive on $Z(f/\con_tf)$ and hence $f/\con_tf$ is irreducible (since $a_0\neq 0$ it is separable). Summing over all possible $h,c,k$ and noting that
		$$\sum_{h,c,k}\lim_{n\to\infty}\Prob(f\in\mathcal F_{n,h,c,k}\mid a_0\neq 0)=1$$
		(here $f$ is a random $q$-additive polynomial as in the statement of Theorem \ref{thm:1}) we conclude
		$$\lim_{n\to\infty}\Prob(f/\con_tf\mbox{ is irreducible}\mid a_0\neq 0)=1.$$
		
		It remains to deal with the case $a_0=0$. In this case we can write $f(t,X)=f_1(t,X^q)$. Then $\con_tf=(\con_tf_1)(X^q)$ and if $a_1\neq 0$ we can apply the above to $f_1$ and use Uchida's irreducibility criterion \cite[Lemma 12.4.1]{FrJa08} to conclude that $f/\con_tf$ is asymptotically almost surely irreducible. If $a_0=a_1=0$ we can iterate this procedure (with probability arbitrarily close to 1 we have $a_i\neq 0$ for some bounded $i$).

	\end{proof}

	%\section{Other models}\label{app:largedegreelargefield}
	\section{The large box model}\label{sec:largebox}
	The large box model considers the same problem of a $q$-additive polynomial with random coefficients taken from $\FF_q[t]_{\le d}$, this time fixing $n$ and $q$ and letting $d \rightarrow \infty$. The proof that the Galois group is almost always $\GL_n(q)$ is short, and mostly relies on a quantitative Hilbert Irreducibility Theorem combined with a result of Dickson \cite{Dickson_1911}, the proof of which we include for completeness.
	
	\begin{theorem}\label{large_box}
		Fix $n>0$ and $q$ a prime power. 
		Let $a_0,\ldots a_{n-1}$ be independent random variables, taking values in $\FF_q[t]_{\le d}$ uniformly. Let  $f = X^{q^n} + a_{n-1} X^{q^{n-1}} + \cdots + a_0 X$ and put $G_f$ the Galois group of $f$ over $\FF_q(t)$. Then 
		\[
		\lim_{d\to \infty} {\rm Prob} (G_f =\GL_n(q)) = 1.  
		\]
	\end{theorem}
	\begin{proof}
		In this case, we consider the generic polynomial $F(A_0, ..., A_{n-1};X) = X^{q^n} + A_{n-1} X^{q^{n-1}} + \cdots + A_0 X\in \FF_q(t)(A_0, ..., A_{n-1})[X]$ where $q$ is a power of the prime $p$ and $A_0,\ldots,A_n$ are independent variables. By Proposition~\ref{prop:basic}(i) the roots $Z(F)$ of $F$ form an $\F_q$-linear subspace of $\overline{\FF_q(t)(A_0, ..., A_{n-1})}$, and since $|Z(F)|=q^n$ ($F$ is separable because $A_0\neq 0$), the dimension of $Z(F)$ is $n$. Pick a basis $e_1, ..., e_n$ for $Z(F)$.
		
		By \cite{Dickson_1911} (see also \cite[Theorem~1.2]{wilkerson1983primer}) the generic Galois group $G_F$ of $F$ over $\FF_q(t)$ is $\GL_n(q)$. The proof is short: the splitting field of $F$ over $\FF_q(t)(A_0,\ldots,A_{n-1})$ is $L=\FF_q(t)(e_1, ..., e_n)$ and it has transcedence degree $n$ over $\FF_q(t)$, hence the $e_i$ are algebraically independent over $\FF_q(t)$ and so $\GL_n(q)\leqslant\mathrm{Aut}(L/\F_q(t))$ (via the action on $e_1,\ldots,e_n$). Now $\GL_n(q)$ permutes $Z(F)$, hence \[\FF_q(t)(A_0, ..., A_{n-1})\subseteq \FF_q(t)(e_1, ..., e_n)^{\GL_n(q)}\] and $\GL_n(q)\le G_F$. Conversely, since $F$ is additive we have $G_F\leqslant\GL_n(q)$, so $G_F=\GL_n(q)$.
		
		%Let $K:=\F_q(u)$ be the field of rational functions in $u$ with coefficients in the finite field $\F_q$. Let $\textbf{T}:=(T_1, ..., T_n)$ and take $f(\textbf{T}, X,Y)=X^{q^n}+T_1 X^{q^{n-1}}+\cdots + T_n X -Y=\prod(X-\sum_{i=1}^{n}a_iU_i)\in K[\textbf{T}][X,Y]$ where $a_i\in\F_q$, and write $f_0:=f(\textbf{T},X,0)$. For a non-negative integer $d$ define $M_d$ to be the set of monic polynomials of degree $d$ \[M_d:=\{t(u) \in K \mbox{ monic polynomial} \mid \deg(t)=d\}.\]
		
		For a given tuple $(a_0, ..., a_{n-1})\in \FF_q[t]^n$ we denote by $f(X)=F(a_0, ..., a_{n-1})(X)=X^{q^n} + a_{n-1} X^{q^{n-1}} + \cdots + a_0 X$ the specialization of $F$ at $(a_0, ..., a_{n-1})$. We suppress the $a_i$ and write $G_f$ for the Galois group of $f$ over $\FF_q(t)$. Note that almost all $f$ are separable (as $d\to\infty$) because $a_0\neq 0$. In this case is well known that $G_f \leqslant G_F$ (since $f$ is a specialization of $F$). Denote by $B(d)$ the number of tuples $(a_0, ..., a_{n-1})\in \FF_q[t]_{\le d}^n$ such that $G_f\ncong G_F$. In the large box model, we fix $n$ and $q$ and let $d \rightarrow \infty$.
		
		By the quantitative Hilbert's Irreducibility Theorem for $\F_q(t)$ \cite[Corollary~3.5]{barysoroker2019explicit}, \[\frac{B(d)}{q^{dn}}=O\left(dq^{-d/2}\right) \mbox{ for }n, q\mbox{ fixed}, d\rightarrow \infty.\]
		
		Therefore, for almost all tuples $(a_0, ..., a_{n-1})\in \FF_q[t]_{\le d}^n$, the Galois group $G_f=\GL_n(q)$ in the large box model.
	\end{proof}
	
	\section{The large finite field model}\label{sec:large q}
	Let $a_0,\ldots a_{n-1}$ be independent random variables taking values in $\FF_q[t]_{\le d}$ uniformly. Consider the random polynomial $f = X^{q^n} + a_{n-1} X^{q^{n-1}} + \cdots + a_0 X$. We aim to fix $n, d$ and describe the Galois group of $f$ over $\FF_q(t)$ as $q \to \infty$, and our asymptotic notation in this section should be read in the $q \rightarrow \infty$ regime (with $n,d$ fixed). The main ingredients of the proof are the following propositions.
	
	Firstly, we assert that we can almost always find specializations with characteristic polynomials of any desired factorization type.
	
	\begin{prop}\label{prop:spec_fact} Let $n,d$ be fixed natural numbers. Then for all but $O\left(q^{n(d+1)-1}\right)$ of the tuples $(a_0,\ldots,a_{n-1})\in\left(\F_q[t]_{\le d}\right)^n$ and any natural numbers $n_1,\ldots,n_k$ with $\sum_{i=1}^kn_i=n$ there exists $\tau\in\F_q$ such that the polynomial
		$${\tilde f}_\tau=X^n+a_{n-1}(\tau)X^{n-1}+\ldots+a_0(\tau)\in\F_q[X]$$ factors into distinct irreducible polynomials of degree $n_1,\ldots,n_k$.\end{prop}
	
	We will prove this proposition in Appendix~\ref{sect:spec_fact}.
	
	We will need an analogue of Proposition~\ref{prop:irreducible subgroups} for the large $q$ regime. In this case we will need a slight modification of the result. Instead of classes $\mathcal{C}_2$ and $\mathcal{C}_3$ we consider the following sets: \begin{align*}\mathcal{D}_2&=\left\{M \setminus H : M= H\wr S_l \in \mathcal{C}_2 \mbox{ for some }l \mid n \mbox{ and } H \cong \GL_{n/l}(q^l)\right\} \\ \mathcal{D}_3&=\left\{M \setminus H : M= H.b \in \mathcal{C}_3 \mbox{ for some prime }b \mid n \mbox{ and } H \cong \GL_{n/b}(q^b)\right\}\end{align*}
	
	\begin{prop}\label{prop:small_n_count_polys}
		
		Let $\CJ=\CD_2\cup \CD_3 \cup \bigcup_{i=4}^8\mathcal C_i \cup \CS.$ The set of characteristic polynomials of elements in $\bigcup\CJ$ has size $o_n\left(q^{n}\right)$.
		
	\end{prop}
	We will prove this proposition in Appendix~\ref{sec:subgroups large q}.
	
	Assuming the propositions, we are able to describe the Galois group in the limit.
	
	\begin{theorem}\label{large_field}
		Fix $d, n>0$. Let $q$ be a prime power. Let $a_0,\ldots a_{n-1}$ be independent random variables, taking values in $\FF_q[t]_{\le d}$ uniformly. Let $f = X^{q^n} + a_{n-1} X^{q^{n-1}} + \cdots + a_0 X$ and let $G_f$ be the Galois group of $f$ over $\FF_q(t)$. Then
		\[
		\lim_{q\to \infty} {\rm Prob} (G_f = \GL_n(q)) = 1.  
		\]
	\end{theorem}
	\begin{proof}
		
		First consider the case where $n=1$. Then $f=X^q+a_0(t)X$ has Galois group $G_f \leqslant \FF_q^\times$. By Proposition 6.1, the subgroup is proper only if $a_0(t)=b(t)^r$ for some $r|q-1$. So the number of choices for $a_0\in \FF_q[t]_{\le d}$ which give $G_f \lneqq \FF_q^\times$ is at most $\sum_{r \mid q-1}q^{d/r+1}\le q^{o(1)}q^{d/2+1}$ since the number of divisors of $q-1$ is $q^{o(1)}$ (see, for example, \cite[Theorem 2.11]{MoVa06} for the analog over $\mathbb Z$). Therefore \[
		{\rm Prob} (G_f = \GL_1(q)) \ge 1-q^{-d/2+o(1)}.
		\]

		Now assume $n>1$. By Proposition~\ref{prop:det}, if $a_0$ is not a nontrivial power of another polynomial (which holds with probability $\to 1$ as $q\to\infty$) then $\det(G_f)=\F_q^\times$. Hence it is enough to show that $\Prob(G_f\supseteq\SL_n(q))\to 1$. 
		
		By Proposition~\ref{prop:spec_fact} and Proposition~\ref{prop: char poly}(ii), with probability at least $1-O\left(1/q\right)$, for any natural numbers $n_1, ..., n_k$ with $\sum_{i=1}^k n_i=n$ there is some $\sigma \in G_f$ such that the characteristic polynomial of $\sigma$ factorizes into irreducible polynomials of degrees $n_1, ..., n_k$.
		
		In particular, with probability at least $(1-O\left(1/q\right))^2$, there are elements $\sigma_{(n)}, \sigma_{(n-1,1)} \in G_f$ such that the characteristic polynomial of $\sigma_{(n)}$ is irreducible and the characteristic polynomial of $\sigma_{(n-1,1)}$ is the product of a linear factor and an irreducible polynomial.
		
		Characteristic polynomials of elements in reducible subgroups are reducible polynomials, so $\sigma_{(n)}$ is not contained in any member of $\mathcal{C}_1$. Similarly characteristic polynomials of elements in $\GL_{n/l}(q)^l$ factor into $l$ factors of degree $n/l$, and therefore $\sigma_{(n)}$ is not contained in the identity coset of any member of $\mathcal{C}_2$. Therefore with probability at least $1-O(1/q)$ the Galois group $G_f$ is not contained in any member of $\mathcal{C}_1$ or identity coset of a $\mathcal{C}_2$-subgroup.
		
		By \cite[Lemma~5.4]{FulmanGuralnick4}, the characteristic polynomials of elements in a conjugate of $\GL_{n/b}(q^b)$ are those whose irreducible factors with degree not divisible by $b$ appear with multiplicity a multiple of $b$. Therefore $\sigma_{(n-1,1)}$ cannot be contained in any conjugate of $\GL_{n/b}(q^b)$ unless $n=b=2$. In this case, every reducible characteristic polynomial of an element of $\GL_1(q^2)$ is inseparable, while $\sigma_{(1,1)}$ has separable characteristic polynomial. So with probability at least $1-O(1/q)$ the Galois group $G_f$ is not contained in the identity coset of any member of $\mathcal{C}_3$.
		
		The remaining elements of maximal subgroups of $\GL_n(q)$ not containing $\SL_n(q)$ are dealt with by Proposition~\ref{prop:small_n_count_polys}, which shows they contribute $o(q^n)$ characteristic polynomials as $n$ is fixed and $q \rightarrow \infty$. Therefore the probability that the Galois group $G_f$ is contained in any member of $\CD_2 \cup \CD_3 \cup \bigcup_{i=4}^8 \mathcal{C}_i \cup \CS$ is $o(1)$.
	\end{proof}
	
	\appendix
	\section{Specializations with prescribed factorization type}\label{sect:spec_fact}
	
	The present subsection is devoted to proving Proposition~\ref{prop:spec_fact}. The main ingredients in the proof are the Chebotarev Density Theorem in function fields and a Bertini-type theorem on specialization of Galois groups which is a consequence of Noether's irreducibility theorem.
	
	\begin{prop}[Chebotarev Density Theorem in function fields, the case $G=S_n$]\label{prop:cheb_sn}
		Let $F\in\F_q[t,X]$ be separable in $X$ and assume $\Gal(F/\overline{\F_q}(t))=S_n$. Let $n_1,\ldots,n_k$ be natural numbers such that $n=\sum_{i=1}^kn_k$. Then the number of $\tau\in\F_q$ such that $F_\tau=F(\tau,X)\in\F_q[x]$ factors into distinct irreducible polynomials of degree $n_1,\ldots,n_k$ equals
		$\frac c{n!}q+O_{\deg F}\left(q^{1/2}\right)$, where $c$ is the number of permutations in $S_n$ with cycle structure $(n_1,\ldots,n_k)$.
	\end{prop}
	
	\begin{proof} This follows from the Chebotarev Density Theorem in function fields [FJ08, 6.4.8], since the degrees in the factorization of $F_\tau$ are precisely the cycle structure of $\Fr_q$ acting on $Z(F_\tau)$.\end{proof}

	We recall some basic facts about Galois covers of varieties. For background on the topic (in the more general setting of schemes) see \cite[\S 5]{Sza09}. We work over a fixed algebraically closed field $k$. Let $V$ be an irreducible variety and $\phi:Y\to V$ a finite \'etale map of degree $n$. We may form the reduced $n$-fold fibered product
	$$Y^{(n)}_V=\{(y_1,\ldots,y_n)\in Y\times_V\cdots\times_VY:\, y_i\neq y_j\mbox{ if }i\neq j\}.$$
	Each connected component $W$ of $Y_V^{(n)}$ is Galois over $V$, $\Gal(W/V)\cong\Gal(k(W)/k(V))$, $k(W)$ is isomorphic to the compositum of the Galois closures (over $k(V)$) of the function fields of components of $Y$ and $|\Gal(W/V)|=\frac{n!}{c(Y^{(n)}_V)}$, where $c(Y_V^{(n)})$ is the number of components of $Y^{(n)}_V$. $\Gal(W/V)$ acts on $W\subset Y^{(n)}_V$ by permutations of the coordinates $(y_1,\ldots,y_n)$ and can be identified with a subgroup of $S_n$.
	
	If in addition $V$ is normal, $Z\subset V$ is an irreducible closed subvariety and $W'=W\cap\pi^{-1}(Z)$ is irreducible ($\pi:Y^{(d)}_V\to V$ is the standard projection) then $W'/Z$ is Galois and there is a natural isomorphism $\Gal(W'/Z)\xrightarrow{\sim}\Gal(W/V)$. Conversely, if $W'$ is reducible then each of its components has a strictly smaller Galois group over $Z$.
	
	In particular if $V$ is affine, normal and irreducible, $k[V]$ its coordinate ring and $F=\sum_0^nc_iX^i\in k[V][X]$ a polynomial with $c_n\Disc(F)\in k[V]^\times$, then $Y=\mathrm{Spec}\,k[V][X]/F\to \mathrm{Spec}\,k[V]=V$ is a finite \'etale cover and each irreducible component $W$ of $Y_V^{(n)}$ is Galois over $V$ with $\Gal(W/V)\cong\Gal(F/k(V)).$ If $Z\subset V$ is closed and irreducible, $F_Z=\sum_0^n c_i|_ZX^i\in k[Z][X]$ and $W'=\pi^{-1}(Z)\cap W$ is irreducible then $W'/Z$ is Galois with $$\Gal(F_Z/k(Z))\cong\Gal(W'/Z)\cong\Gal(W/V)\cong\Gal(F/k(V)).$$
	Conversely, if $\Gal(F_Z/k(Z))\cong\Gal(F/k(V))$ then $W'$ is irreducible.
	
	Now let $V,F,W$ be as above. We may identify $$Y^{(n)}_V=\{(v,x_1,\ldots,x_n)\in V\times\AA^n:F(v,x_i)=0,x_i\neq x_j \mbox{ if }i\neq j\}$$ and then $\Gal(W/V)\leqslant S_n$ acts on $W\subset Y^{(n)}_V$ by permuting $x_1,\ldots,x_n$. Also denote by $\xi_i:W\to k$ the projection to the $x_i$ coordinate. Note that $\xi_1,\ldots,\xi_n$ are precisely the roots of $F$ in $k(W)$. Let $\bflam=(\lambda_1,\ldots,\lambda_n)\in k^n$ and assume that for some (and therefore for a dense open subset of) $(v,x_1,\ldots,x_n)\in W$  we have \begin{equation}\label{eq:lam cond}\sum_{i=1}^n\lambda_ix_{\sigma(i)}\neq\sum_{i=1}^n\lambda_ix_i\quad\mbox{for all}\quad 1\neq\sigma\in\Gal(W/V).\end{equation} Then \begin{equation}\label{eq:resolvent}H_{F,W,\bflam}=\prod_{\sigma\in\Gal(W/V)}\left( X-\sum_{i=1}^n\lambda_i\xi_{\sigma(i)}\right)\in k[V][X]\end{equation} is called a resolvent for $F$ (the coefficients are in $k[V]$ because $k[W]^{\Gal(W/V)}=k[V]$).
	
	The resolvent $H=H_{F,W,\lambda}$ is irreducible and separable (by (\ref{eq:lam cond})). If additionally $\Disc(H)\neq 0$ on all of $V$ (i.e. the condition (\ref{eq:lam cond}) holds everywhere) then $W\cong\Spec k[V][X]/H$ (as $V$-covers, isomorphism given by $(v,x_1,\ldots,x_n)\mapsto (v,\textstyle\sum_{i=1}^n \lambda_i x_i)$) and for any closed irreducible $Z\subset V$ we have $W'=\pi^{-1}(Z)\cap W\cong\Spec k[Z][X]/H_Z$ ($H_Z=\textstyle\sum h_i|_ZX^i\in k[Z][X]$, where $H=\textstyle\sum h_iX^i$). Hence $\Gal(W'/Z)=\Gal(W/V)$ iff $H_Z$ is irreducible (both equivalent to $W'$ being irreducible).
	
	\begin{prop}[Bertini-Noether-type irreducibility theorem for Galois groups] Let $k$ be an algebraically closed field, $\mathbf A=(A_1,\ldots,A_m),t,X$ variables, $F\in k[\bfA,t,X]$ separable of degree $n$ in $X$ and assume that for some $\bfa\in k^m$ with $F(\bfa,t,X)$ separable of degree $n$ in $X$ we have $\Gal(F(\bfa,t,X)/k(t))\cong\Gal(F(\bfA,t,X)/k(\bfA,t))$. Then
		$$S=\left\{\bfa\in\AA^m_k:F(\bfa,t,X)\mbox{ separable of degree }n\mbox{ in } X,\,\Gal(F(\bfa,t,X)/k(t))\not\cong\Gal(F/k(\bfA,t))\right\}$$ is a Zariski closed subset of $\mathbf \AA^m_k$ and contained in a hypersurface of degree $O_{m,\deg F}(1)$ ($\deg F$ is the total degree in $\bfA,t,X$).\end{prop}
	
	\begin{proof} Write $F=\sum_{i=0}^n c_i(\bfA,t)X^i\in k[\bfA,t][X],\,c_n\neq 0$. Since $f$ is separable $c_n\Disc(F)\neq 0$. Replacing $F$ with $c_n^{n-1}F(X/c_n)$ we assume WLOG that $c_n=1$. Consider $$V=\{(\bfa,\tau)\in\AA^{m+1}:\Disc_XF(\bfa,\tau,X)\neq 0\},$$
		$$Y=\{(\bfa,\tau,x)\in\AA^{m+2}:\bfa\in V,F(\bfa,\tau,x)=0\}.$$ The projection $Y\to V:(\bfa,\tau,x)\mapsto(\bfa,\tau)$ is finite \'etale of degree $n$. Note that $Y,V$ are affine and $V$ is smooth and irreducible. Let $W$ be an irreducible component of $Y^{(n)}_V$ and $\xi_1,\ldots,\xi_n:W\to k$ the projections defined in the discussion preceding the proposition (equivalently the roots of $F$ in $k(W)$)
		
		Now choose a resolvent $H=H_{F,W,\bflam}\in k[V][X]=k[\bfA,t]_{\Disc(F)}[X]$ for $F$ such that \\$\Disc_X(H)(\bfa_0,\tau_0)\neq 0$ (a suitable choice of $\bflam$ satisfying (\ref{eq:lam cond}) is possible because $k$ is infinite). By (\ref{eq:resolvent}), the assumption that $F$ is monic and the fact that $k[\bfA,t]$ is integrally closed, we have in fact $H\in k[\bfA,t,X]$. Since $H|\prod_{\sigma\in S_n}(X-\textstyle\sum_{i=1}^n\lambda_i\xi_{\sigma(i)})$ we also see (by expanding the RHS and writing each coefficient of $X^i$ as a polynomial in the elementary symmetric functions in $\xi_1,\ldots,\xi_n$) that $\deg H,\deg\Disc_X(H)=O_{m,\deg F}(1)$. Denote $U=V\setminus \{\Disc_X(H)=0\}$, $P:\AA^{m+1}\to\AA^m$ the projection $P(\bfa,\tau)=\bfa$. By the discussion preceding the proposition we know that for $\bfa\in P(U)$ we have $\Gal(F(\bfa,t,X)/k(t))=\Gal(F/k(\bfA,t))$ iff $H_{U\cap(\bfa\times\AA^1)}=H(\bfa,t,X)\in k[t,X]$ is irreducible. We also know that $\bfa_0\in P(U)$.
		
		Write $H=\sum h_{ij}(\bfA)t^iX^j$. By the Noether irreducibility theorem (apply \cite[Theorem 5.3.1]{Gey13} to the homogenized version of $H$ w.r.t. the variables $t,X$ and use the fact that $H$ is monic in $X$) there exists polynomials $g_1,\ldots,g_N\in\mathbb Z[X_{ij}]$ with $N,\deg g_l=O_{\mathrm{deg}_{t,X}H}(1)$ such that $H(\bfa,t,X)$ is reducible iff $g_1(h_{ij}(\bfa))=\ldots=g_N(h_{ij}(\bfa))=0$. Hence $g_l(h_{ij}(\bfa_0))\neq 0$ for some $l$ and we assume WLOG that $g_1(h_{ij}(\bfa_0))\neq 0$. 
		
		Setting $\tilde g=g_1(h_{ij}(\bfA))\in k[\bfA]$ we have $\deg\tilde g=O_{m,\deg F}(1)$ and for $\bfa\in P(U)$ such that $\tilde g(\bfa)\neq 0$ we have $H(\bfa,t,X)$ irreducible and hence $\Gal(F(\bfa,t,X)/k(t))=\Gal(F/k(\bfA,t))$. This shows that $P(V)\setminus S=\{\bfa\in P(V): \Gal(F(\bfa,t,X)/k(t))=\Gal(F/k(\bfA,t))\}$ is open (since we can run the above argument starting from any $\bfa_0\in P(V)\setminus S$).
		
		Now write $\Disc_XF\cdot\Disc_XH=\sum d_i(\bfa)t^i$ and let $i$ be such that $d_i\neq 0$. Note that $P(U)\supset\{d_i\neq 0\}$ and hence $S\subset\{\bfa\in\AA^m:\tilde g(\bfa)d_i(\bfa)=0\}$ is contained in a hypersurface defined by an equation of degree $O_{m,\deg F}(1)$.
		
	\end{proof}
	
	\begin{prop}\label{prop:spec_ff} Let $\mathbf A=(A_1,\ldots,A_m),t,X$ be variables, $F\in\F_q[\bfA,t,X]$ separable in $X$. Assume that for some $\bfa\in\overline{\F_q}^n$ with $F(\bfa,t,X)$ separable of degree $n$ in $X$ we have\\ $\Gal(F(\bfa,t,X)/\overline{\F_q}(t))\cong\Gal(F/\F_q(\bfA,t))$. Then for all but $O_{m,\deg F}(q^{m-1})$ tuples $\bfa\in\F_q^m$ we have that $F(a,t,X)$ is separable of degree $n$ in $X$ and $\Gal(F(\bfa,t,X)/\F_q(t))\cong\Gal(F/\F_q(\bfA,t))$.
	\end{prop} 
	
	\begin{proof} First note that since $\Gal(F(\bfa,t,X)/\overline{\F_q}(t))\hookrightarrow\Gal(F/\overline{\F_q}(\bfA,t))\hookrightarrow\Gal(F/\F_q(\bfA,t))$ the assumption implies $\Gal(F(\bfa,t,X)/\F_q(t))\cong\Gal(F/\overline{\F_q}(\bfA,t))$. Apply the previous proposition with $k=\overline{\F_q}$. The set of $\bfa\in\F_q^m$ not satisfying the assertion is precisely $S(\F_q)$ in the notation of the previous proposition. Since $S$ is contained in a hypersurface of degree $O_{m,\deg F}(1)$ the assertion follows from \cite[Lemma 1]{LaWe54}.
	\end{proof}
	
	\begin{proof}[Proof of Proposition~\ref{prop:spec_fact}] Let $n,d$ be fixed natural numbers, and consider $$F[\bfA, X]=X^n+\sum_{i=0}^{n-1}a_i(t)X^i\in\F_q[\bfA,t,X]$$ where $a_i=\sum_{j=0}^dA_{ij}t^j,0\le i\le n-1,0\le j\le d$ and $\bfA=(A_{ij})$ are independent variables. If we find a single $f\in\overline{\F_q}[X]$ such that $\Gal(f(X)-t/\overline{\F_q}(t))=S_n$, then since $\Gal(F/\overline{\F_q}(\bfA,t))\leqslant S_n$ we have $\Gal(F/k(\bfA,t))\cong \Gal(f(X)-t/\overline{\F_q}(t))\cong S_n$ and Proposition~\ref{prop:spec_ff} shows that for all but $O\left(q^{n(d+1)-1}\right)$ specializations of $\bfa$ of $\bfA$ we have $\Gal(F(\bfa,t,X)/\F_q(t))=S_n$. For any such $\bfa$, Proposition~\ref{prop:cheb_sn} shows that a specialization $F(\bfa,\tau,X),\tau\in\F_q$ with the required factorization exists if $q$ is large enough. It remains to produce an $f$ as above for every $n\ge 1$ and prime power $q$.
		
		Assuming $n\ge 3$ (otherwise one of $f=X,X^2,X^2+X$ works), let $h\in\overline{\F_q}[X]$ be separable of degree $n-2$ with $h(0)\neq 0$ such that $f=X^2h(X)$ is indecomposable, i.e. cannot be written $f=u\circ v,\,\deg u,\deg v\ge 2$. An $h$ as above exists for $n\ge 3$ by a simple dimension count. Then standard methods for computing Galois groups (see e.g. \cite{BiSw59}) show that $\Gal(f(X)-t/\overline{\F_q}(t))\leqslant S_n$ is primitive (because $f$ is indecomposable) and contains a transposition (because of the simple ramification of $f$ over $0$), hence is $S_n$ by \cite[Theorem 13.3]{Wie64}, as required.
		
	\end{proof}
	
	\section{Counting characteristic polynomials when $q$ is large}\label{sec:subgroups large q}
	
	This appendix is devoted to the proof of Proposition~\ref{prop:small_n_count_polys}. There are at most $q^n$ conjugacy classes in $\GL_n(q)$ \cite[Lemma~5.9]{Maslen1997}. Regular semisimple elements of $\GL_n(q)$ are semisimple elements with pairwise distinct eigenvalues, and since this property is preserved by conjugation we refer to conjugacy classes of regular semisimple elements as regular semisimple conjugacy classes. By \cite[Theorem~2.2]{fulman2013number} the number of conjugacy classes in $\GL_n(q)$ which are not regular semisimple is at most $q^n-\frac{q-1}{q+1}(q^n-(-1)^{n})=o_n(q^n)$, so we may restrict ourselves to regular semisimple conjugacy classes.
	
	It is standard that if $C$ is a regular semisimple conjugacy class then $|C|>a|\GL_n(q)|/q^n$ for some $a>0$ depending only on $n$ and not on $q$. Now let $\Delta$ be the set of all elements of $\GL_n(q)$ which are regular semisimple and not contained in $\bigcup \CJ$. This is the same $\Delta$ as in \cite{Garzoni}, and for more information we direct the reader to a textbook on linear algebraic groups, for example \cite{MalleTesterman}. By \cite[Theorem~3.6]{Garzoni}, $|\Delta|=|\GL_n(q)|(1-O_n\left(1/q\right))$.
	
	Let $J$ be the set of all regular semisimple elements of $\bigcup \CJ$ and let $J'$ be the set of regular semisimple conjugacy classes of $\GL_n(q)$ which intersect some member of $\CJ$. Then using \cite[Theorem~3.6]{Garzoni} and the lower bound on the size of a regular semisimple conjugacy class we see that \[\frac{|J'|}{q^n} < \frac{|J|}{|\GL_n(q)|}=O_n(1/q)\] whence $|J'|=o_n(q^n)$. Therefore the number of characteristic polynomials of elements in $\bigcup \CJ$ is $o_n(q^n)$.

	\bibliography{bib}
	\bibliographystyle{alpha}
\end{document}